\documentclass [a4paper, 12pt]{amsart}

\usepackage{times}

\usepackage{amsmath,amsfonts,amsthm,amssymb,amscd}
\usepackage{graphicx, epsfig, psfrag}
\usepackage{pb-diagram}
\usepackage{hyperref}
\usepackage{wasysym}
\usepackage{epstopdf,yfonts}
\usepackage{fancyhdr}
\usepackage{tikz}
\usepackage[all]{xy}
\usepackage{color}
\usepackage{xargs}
\usepackage{comment}   
\usepackage[sc]{mathpazo}
\usepackage{eulervm}
\usepackage{float}
\usepackage[marginpar=2.4cm]{geometry}

\theoremstyle{plain}
\newtheorem{cor}{Corollary}[section]
\newtheorem{thm}[cor]{Theorem}
\newtheorem{prop}[cor]{Proposition}
\newtheorem{q}[cor]{Question}
\newtheorem{claim}[cor]{Claim}

\newtheorem{lem}[cor]{Lemma}
\theoremstyle{definition}

\newtheorem{defn}[cor]{Definition}
\theoremstyle{remark}
\newtheorem{remark}[cor]{Remark}

\newtheorem{ex}[cor]{Example}
\newtheorem{example}[cor]{Example}

\newcommand\bqn{\begin{equation*}}
\newcommand\eqn{\end{equation*}}
\newcommand\bq{\begin{equation}}
\newcommand\eq{\end{equation}}
\newcommand\be{\begin{enumerate}}
\newcommand\ee{\end{enumerate}}
\newcommand\bei{\begin{itemize}}
\newcommand\eei{\end{itemize}}
\newcommand\ba{\begin{aligned}}
\newcommand\ea{\end{aligned}}
\newcommand\ban{\begin{aligned*}}
\newcommand\ean{\end{aligned*}}
\newcommand{\bsm}{\left(\begin{smallmatrix}}
\newcommand{\esm}{\end{smallmatrix}\right)}                   
\newcommand{\bpm}{\begin{pmatrix}}
\newcommand{\epm}{\end{pmatrix}}

\newcommand{\thismonth}{\ifcase\month 
  \or January\or February\or March\or April\or May\or June%
  \or July\or August\or September\or October\or November%
  \or December\fi}

\newcommand{\calA}{\mathcal A}
\newcommand{\calB}{\mathcal B}

\newcommand{\calG}{\mathcal G}
\newcommand{\calH}{\mathcal H}

\newcommand{\calL}{\mathcal L}

\newcommand{\calO}{\mathcal O}

\newcommand{\calR}{\mathcal R}
\newcommand{\calS}{\mathcal S}
\newcommand{\calT}{\mathcal T}

\newcommand{\calV}{\mathcal V}

\newcommand{\calX}{\mathcal X}

\newcommand{\FF}{\mathbb F}

\newcommand{\PP}{\mathbb P}

\newcommand{\RR}{\mathbb R}

\newcommand{\bC}{\mathbf C}

\newcommand{\bF}{\mathbf F}

\newcommand{\bK}{\mathbf K}
\newcommand{\bL}{\mathbf L}

\newcommand{\bN}{\mathbf N}

\newcommand{\bQ}{\mathbf Q}
\newcommand{\bR}{\mathbf R}

\newcommand{\bU}{\mathbf U}

\newcommand{\bZ}{\mathbf Z}

\newcommand{\SL}{\operatorname{SL}}
\newcommand{\Sp}{\operatorname{Sp}}
\newcommand{\PSp}{\operatorname{PSp}}
\newcommand{\PSL}{\operatorname{PSL}}
\newcommand{\PO}{\operatorname{PO}}

\newcommand{\GL}{\operatorname{GL}}

\newcommand{\PU}{\operatorname{PU}}
\newcommand\Sym{\mathrm{Sym}}

\newcommand{\Hom}{\mathrm{Hom}}

\newcommand\diag{\mathrm{diag}}
\newcommand\Id{\mathrm{Id}}
\newcommand{\Isom}{\mathrm{Isom}}
\newcommand\ov{\overline}
\newcommand\per{\mathrm{per}}

\newcommand{\sign}{\operatorname{sign}}

\newcommand\tr{\mathrm{tr}}

\renewcommand{\>}{\rangle}
\newcommand{\supp}{\mathrm{supp}} 

\renewcommand{\H}{{\mathcal H}^2} 
\newcommand{\pr}{{\rm pr}} 
\newcommand{\deH}{\partial \H}

\newcommand\Syst{\mathrm{Syst}}

 \newcommand{\Oo}{\mathcal O} 
 \newcommand{\Ii}{\mathcal I} 

\newcommandx{\ioo}[2]{I_{(#1,#2)}}
\newcommandx{\ioc}[2]{I_{(#1,#2]}}
\newcommandx{\ico}[2]{I_{[#1,#2)}}
\newcommandx{\icc}[2]{I_{[#1,#2]}}

\newcommand{\hg}{{H_\Gamma}}

\newcommand{\calCR}{\mathcal{CR}}
 \newcommand{\Curr}{\mathcal C} 
 \newcommand{\bRr}{\mathcal R} 
\newcommand{\snr}{\calS^n_\bR} 
\newcommand\snf{\calS^n_\bF} 
\newcommand{\Dd}{\mathcal D} 
\newcommand{\Yy}{\mathcal Y} 
 \newcommand{\p}{{\rm p}} 

\newcommand{\g}{\gamma}

\newcommand{\Rec}{\calR} 
\newcommand{\ovcirc}{\mathring}
\newcommand{\cR}[1]{r(#1)} 
\newcommand{\cRo}{r} 

\makeindex

\date{\today}

\begin{document}
\title[Crossratios, barycenters and maximal representations]{Positive crossratios, barycenters, trees and applications to maximal representations}

\author[]{M. Burger}
\address{Department Mathematik, ETH Zentrum, 
R\"amistrasse 101, CH-8092 Z\"urich, Switzerland}
\email{burger@math.ethz.ch}

\author[]{A. Iozzi}
\address{Department Mathematik, ETH Zentrum, 
R\"amistrasse 101, CH-8092 Z\"urich, Switzerland}
\email{iozzi@math.ethz.ch}

\author[]{A. Parreau}
\address{Institut Fourier, CS 40700, 38058 Grenoble cedex 09, France}
\email{Anne.Parreau@univ-grenoble-alpes.fr}

\author[]{M. B. Pozzetti}
\address{Mathematical Institute, Heidelberg University, Im Neuenheimerfeld 205, 69120 Heidelberg, Germany }
\email{pozzetti@mathi.uni-heidelberg.de}
\thanks {
Marc Burger and Alessandra Iozzi were partially supported by the SNF grant 2-77196-16. 
Alessandra Iozzi acknowledges moreover support from U.S. National Science Foundation grants DMS 1107452, 1107263, 1107367 
"RNMS: Geometric Structures and Representation Varieties" (the GEAR Network). Beatrice Pozzetti is supported by the Deutsche Forschungsgemeinschaft under Germany’s Excellence Strategy EXC-2181/1 - 390900948 (the Heidelberg STRUCTURES Cluster of Excellence), and acknowledges further support by DFG grant 338644254 (within the framework of SPP2026).
Marc Burger, Alessandra Iozzi and Beatrice Pozzetti  would like to thank the Isaac Newton Institute for Mathematical Sciences, Cambridge, 
for support and hospitality during the programme ``Non-Positive Curvature Group Actions and Cohomology'' 
where work on this paper was undertaken. This work was supported by EPSRC grant no P/K032208/1.}

\date{\today}

\begin{abstract}
We study metric properties of maximal framed representations of fundamental groups of surfaces in symplectic groups over real closed fields, 
interpreted as actions on  Bruhat--Tits buildings endowed with adapted Finsler norms. 
We prove that the translation length can be computed as intersection with a geodesic current, 
give sufficient conditions guaranteeing that such a current is a multicurve, and, if the current is a measured lamination,  
construct an isometric embedding of the associated tree in the building. 
These results are obtained as application of  more general results  of independent interest on positive crossratios and actions with compatible barycenters. 
\end{abstract}
\maketitle


\section{Introduction}\label{sec:intro}
\subsection*{Maximal framed representations in real closed fields}
Let $\Sigma:=\Gamma\backslash\H$ be the quotient of the Poincar\'e upper half plane $\H$ by a torsion-free
lattice $\Gamma<\PSL(2,\bR)$ and let $G$ be a simple real algebraic group. The aim of Higher Teichm\"uller Theory is to single out and study special components
or specific semialgebraic subsets of the representation variety $\Hom(\Gamma,G)$
that consist of injective homomorphisms with discrete image; such components thus generalize the Teichm\"uller space. 
Prominent examples are Hitchin components for  real split  groups $G$ (for example $G=\PSL(n,\bR)$ or $\PSp(2n,\bR)$), 
maximal representations for  Hermitian groups $G$ (for example $G=\PU(p,q)$ or $ \PSp(2n,\bR)$) and $\Theta$-positive representations for $G=\PO(p,q)$. 
The goal of this paper is to study asymptotic properties of maximal representations into $\PSp(2n,\bR)$ with the aid of geodesic currents. 
This applies in particular to the $\PSp(2n,\bR)$-Hitchin component if $\Sigma$ is compact.

In the study of appropriate compactifications of   character varieties,   representations of $\Gamma$ into algebraic groups
over non-Archimedean real closed fields play an important role \cite{Brumfiel, Alessandrini, APcomp, BP, BIPP-ann}: 
for example, ultralimits of representations in $\PSp(2n,\bR)$ can be understood as  representations $\rho_{\omega,\sigma}\colon\Gamma\to\PSp(2n,\bR_{\omega,\sigma})$,
where $\bR_{\omega,\sigma}$ is a Robinson field (see \S\ref{ex:RobF} and \cite{APcomp}). However, the viewpoint of the real spectrum compactification of character varieties leads typically to the study of representations into real closed fields $\bF$ that are small when compared to Robinson fields, for example $\bF$ is often of finite transcendence degree over the field of real algebraic numbers.

Given a representation $\rho\colon\Gamma\to\PSp(2n,\bF)$ where $\bF$ is a general real closed field, 
having a maximal Toledo invariant as defined in \cite[Definition~18]{BIPP-ann}, admitting a maximal framing defined on a $\Gamma$-invariant non-empty subset of $\deH$, 
or admitting a maximal framing defined on the set of fixed points of hyperbolic elements 
are all equivalent conditions \cite{BIPP-RSPECmax} (see \cite[Theorem~20]{BIPP-ann} for a precise statement). 
This paper solely relies upon the third definition, which we now recall.
If $\bF^{2n}$ is endowed with the standard symplectic form, let
$\calL(\bF^{2n})$ be the space of Lagrangians in $\bF^{2n}$.  The
Maslov cocycle (see \S\ref{subsubsec:max-tr-quadr} and \cite{LV}) classifies the orbits of
$\PSp(2n,\bF)$ on triples of pairwise transverse Lagrangians. Such a
triple is {\em maximal} if the cocycle takes its maximal value $n$. Let $\deH$ be the boundary of the hyperbolic plane, which we endow with the cyclic ordering on triples of points induced by the orientation of $\H$.

\begin{defn}  Let $\hg\subset\deH$ be the set of fixed points of hyperbolic elements of $\Gamma$. 
A representation  $\rho\colon\Gamma\to\PSp(2n,\bF)$  is \emph{maximal framed}
if there is a 
$\rho$-equivariant map $\varphi\colon\hg\to\calL(\bF^{2n})$ sending positively oriented triples 
in $\hg$ to maximal triples 
of Lagrangians.
\end{defn}

We assume that the real closed field $\bF$ admits an order compatible $\bR$-valued valuation $v$ 
with value group $\Lambda=v(\bF^\times)<\bR$.  
Then the group $\PSp(2n,\bF)$  acts by isometries on a $\Lambda$-metric space\footnote{In \cite{BIPP-ann} denoted with $\calB_{\PSp(2n,\bF)}$.}
$\calB_n^\bF$ obtained as  quotient of the Siegel upper half space $\snf$ associated to $\PSp(2n,\bF)$ (see \S~\ref{subsec:Siegel} and \cite[\S~3.4]{BIPP-ann}).  
If $\bF=\bR$, $\calB_n^\bR$ coincides with $\snr$, 
while if $\bF$ is non-Archimedean this  metric space $\calB_n^\bF$ sits naturally inside the Bruhat-Tits building associated to $\PSp(2n,\bF)$ as a dense subset.\footnote{Note, however, that the metric that we consider here is only biLipschitz to the restriction of the CAT(0) metric on the Bruhat-Tits building. See \S\ref{subsec:Siegel} for the Finsler metric relevant to our purposes.} 
The latter relationship will play no role in this paper, but  will be discussed in detail in \cite{BIPP-RSPEC} (see also \cite[\S2.1]{BIPP-ann} and \cite{KT1}). 
The translation length of an element $g\in\PSp(2n,\bF)$ acting on $\calB_n^\bF$ induces then the \emph{length function}
\bqn
L(g):=-2\sum_{i=1}^nv(|\lambda_i|)\,,
\eqn
where $\lambda_1,\dots,\lambda_n,\lambda_n^{-1},\dots,\lambda_1^{-1}\in \bF(\sqrt{-1})$ are  the eigenvalues of a representative $\ov g\in\Sp(2n,\bF)$ of $g\in\PSp(2n,\bF)$ counted with multiplicity and ordered in such a way that $|\lambda_1|\geq\dots\geq|\lambda_n|\geq1$. Here we denote by $|\cdot|\colon\bF(\sqrt{-1})\to \bF_+$ the absolute value,
that is the square root of the norm function on the quadratic extension $\bF(\sqrt{-1})$ of $\bF$. 
%
%
\medskip

In our first result we construct a geodesic current on $\Sigma$ encoding the length function $\gamma\mapsto L(\rho(\gamma))$
of a maximal framed representation.
%
Recall that a \emph{geodesic current} is a $\Gamma$-invariant positive Radon measure on the space $(\deH)^{(2)}$ of pairs of distinct points in $\deH$. 
The \emph{Bonahon intersection} $i(\mu,\lambda)$ of two geodesic currents $\mu$ and $\lambda$ extends the topological intersection number
of homotopy classes of curves on $\Sigma$: in fact a (non-oriented) closed geodesic $c\subset\Sigma$ gives rise to the current 
$\delta_c:=\frac{1}{2}\sum_{(a,b)}\delta_{(a,b)}$,
where we sum on the set of oriented geodesics $(a,b)\in (\deH)^{(2)}$ lifting $c$.
For such currents $i(\delta_c,\delta_{c'})$
is the topological intersection number of $[c]$ and $[c']$.

\begin{thm}\label{thm_intro:thm1} Let $\bF$ be a real closed field with an order compatible valuation $v$ and
let  $\rho\colon\Gamma\to\PSp(2n,\bF)$ be maximal framed. 
Then there is a geodesic current $\mu_\rho$ such that, 
for any closed geodesic $c\subset\Sigma$ and for every  $\gamma\in\Gamma$ representing $c$,
\bq\label{eq:intersection=length}
i(\mu_\rho,\delta_c)=L(\rho(\gamma))\,.
\eq
The current $\mu_\rho$ is non-zero if and only if there exists $\gamma\in \Gamma$ with $v(\tr(\rho(\gamma)))<0$.
\end{thm}

If $n=1$, $\bF=\bR$ and $\rho\colon\Gamma\to\PSL(2,\bR)$ is the lattice embedding,
then $\mu_\rho$ is the Liouville current \cite{Bon88-curr}, that is the unique $\PSL(2,\bR)$-invariant geodesic current.
If $\Sigma$ is compact and $\bF=\bR$, Theorem \ref{thm_intro:thm1} was proven by Martone--Zhang \cite[Theorem~1.1]{Martone_Zhang}.
Notice that in the case of $\SL(2,\FF)$ $\mu_\rho$ is always a measured lamination (see \S~\ref{subsec:8.2}).

For the next result we will need the notion of {\em systole} of a maximal framed representation $\rho$ (following \cite{BIPP1})
\bqn
\Syst(\rho):=\inf\;\{L(\rho(\gamma)):
\, \gamma\in\Gamma,
\, \gamma\text{ hyperbolic}\}\,.
\eqn 
The systole of any real maximal framed representation $\rho\colon\Gamma\to\PSp(2n,\bR)$ is positive (see \S~\ref{s.6.2}). 
On the other hand, for non-Archimedean real closed fields $\bF$, many different possibilities can happen: 
 if $\Sigma$ is compact, all  maximal framed representations $\rho\colon\Gamma\to\SL(2,\bF)$ have vanishing systole (since $\mu_\rho$ is a measured lamination), 
 while, in higher rank,  there are many examples of non-Archimedean  maximal framed representations with positive systole (see \cite[Corollary 1.11]{BIPP1}). 
 For these representations we have:
\begin{cor}\label{cor:cor4}  
Assume that $\Sigma$ is compact, and  $\bF$ is  non-Archimedean.
Let $\rho\colon\Gamma\to\PSp(2n,\bF)$ be a maximal framed representation.
If $\Syst_\Sigma(\rho)>0$, then 
for every $x\in \calB_n^\bF$
the orbit map 
\bqn
\ba
\Gamma&\longrightarrow\,\,\calB_n^\bF\\
x&\longmapsto\rho(\gamma)x
\ea
\eqn
is a  quasi-isometric embedding.
\end{cor}

We now give a robust criterion guaranteeing that the current $\mu_\rho$ is atomic. 
We say that a geodesic current is a \emph{multicurve}
 if it is a finite  sum of $\Gamma$-orbits of Dirac masses on (lifts of) closed geodesics and geodesics with endpoints in cusps.
 
\begin{thm}\label{thm_intro:multicurves}
Let $\rho\colon\Gamma\to\Sp(2n,\bF)$ be maximal framed and let $\bQ(\rho)<\bF$ be the field generated over $\bQ$ 
by matrix coefficients of $\rho$.  If the restriction of $v\colon\bF^\times\to\bR$ to $\bQ(\rho)$ is discrete,
 then, up to rescaling, the associated current $\mu_\rho$ is a multicurve.
\end{thm}

Using Strubel coordinates we construct examples of representations to which Theorem \ref{thm_intro:multicurves} applies  (see \S\ref{ex:strubel}).
Moreover we  prove in \cite{BIPP-RSPECmax} that a maximal framed representation
$\rho\colon\Gamma\to\Sp(2n,\bF)$ is always conjugate to a representation
$\rho'\colon\Gamma\to\Sp(2n,\bF_1)$ for a finite extension $\bF_1$ of the field
$\bQ(\tr(\rho))$ generated by the traces of the representation
$\rho$. As a result, Theorem \ref{thm_intro:multicurves} applies as
soon as the field $\bQ(\tr(\rho))$ generated by the traces of the
representation $\rho$ has discrete valuation; using this we  show in
\cite{BIPP-RSPECmax} that multicurves are dense in both the real
spectrum and Weyl chamber length compactifications of  character varieties of maximal representations.

\medskip


We turn now to the case in which  the current $\mu_\rho$ in Theorem~\ref{thm_intro:thm1} is a measured lamination. 
In this case we denote by $\calT(\mu_\rho)$ the associated $\bR$-tree, and by  $\calV(\mu_\rho)$   its vertex set (see \S~\ref{s.treemu} for the definition).
We have then
\begin{thm}
\label{thm_intro:thm5}  
Let $\bF$ be  non-Archimedean real closed and let
 $\rho\colon\Gamma\to\Sp(2n,\bF)$ 
be a maximal framed representation. 
If the associated current $\mu_\rho$ 
is a measured lamination,
then there is a $\Gamma$-equivariant isometric embedding
\bqn
\calV(\mu_\rho)\hookrightarrow{\calB_n^\bF}.
\eqn
\end{thm}
We will see that if $\bQ(\rho)$ has discrete valuation, then $\calV(\mu_\rho)$ is the vertex set of a simplicial tree (see \S~\ref{thm:7.3} for a general statement).

\subsection*{Currents associated to positive crossratios}
The proof of Theorem~\ref{thm_intro:thm1}  relies on an abstract framework
that is applicable to more general situations and that we shortly describe here.
Let $X\subset\deH$ be a $\Gamma$-invariant non-empty subset, such as, for example the set $\hg$ of fixed points
of hyperbolic elements in $\Gamma$ 
and let  $X^{[4]}$ denote the set of positively ordered quadruples in $X$. A {\em positive crossratio} is a $\Gamma$-invariant function 
\bqn
[\,\cdot\,,\,\cdot\,,\,\cdot\,,\,\cdot\,]\colon X^{[4]}\longrightarrow[0,\infty),
\eqn
 that
 is flip-invariant
\bqn
[x_1,x_2,x_3,x_4]=[x_3,x_4,x_1,x_2]
\eqn
and satisfies the property
\bqn
[x_1,x_2,x_4,x_5]=[x_1,x_2,x_3,x_5]+[x_1,x_3,x_4,x_5]\,
\eqn
 whenever defined. Our definition is considerably more general than others existing in the literature (see Remark~\ref{rem:cr} for a  comparison). 
 First we only require our crossratio to be defined  on a dense subset $X$ of the boundary of the hyperbolic plane: 
 this is important for some of the applications. For example, if the field $\bF$ is countable, also $\calL(\bF^n)$ is countable, 
 and thus a maximal framing associated to a representation $\rho\colon\Gamma\to\Sp(2n,\bF)$, 
 as well as the induced crossratio, can only be defined on a countable set. Furthermore some representations, 
 as for example those defined via Fock--Goncharov or shear coordinates,
 only have a framing defined on a countable set not including the fixed points of hyperbolic elements. 
Second we do not require any continuity on our crossratio.
 In many interesting examples, the crossratios arising from representations over non-Archimedean real closed fields are integer valued, 
 and thus cannot be continuous. 
 Dropping the continuity assumption on the crossratio also allows us to  us encompass the theory of crossratios arising from actions on trees (see Example \ref{ex:trees}). 

If $\gamma\in\Gamma$ is hyperbolic and $\{\gamma_-,\gamma_+\}\subset X$,
the {\em period\footnote{See \S~\ref{ssec:periods} for a more general definition of the period without the restriction that $\{\gamma_-,\gamma_+\}\subset X$.} $\per(\gamma)$ of $\gamma$ with respect to $[\,\cdot\,,\,\cdot\,,\,\cdot\,,\,\cdot\,]$} 
is defined by
\bqn
\per(\gamma)\colon=[\gamma_-,x,\gamma x,\gamma_+]\,,
\eqn
where $x\in X$ is any point such that $(\gamma_-,x,\gamma x,\gamma_+)\in X^{[4]}$.
We show:

\begin{thm}\label{thm_intro:thm6}  Let 
$ X\subset\deH$ be a $\Gamma$-invariant non-empty subset 
and $[\,\cdot\,,\,\cdot\,,\,\cdot\,,\,\cdot\,]$ 
a positive crossratio on $ X$.  
Then there is a geodesic current $\mu$ on $\Sigma$ such that for all hyperbolic $\g\in\Gamma$
\bqn
\per(\gamma)=i(\mu,\delta_c).
\eqn
The geodesic current $\mu$ depends continuously on the crossratio $[\,\cdot\,,\,\cdot\,,\,\cdot\,,\,\cdot\,]$.
\end{thm}
The theorem has been previously shown by Martone--Zhang under the hypothesis that $X=\deH$ and the crossratio is continuous \cite{Martone_Zhang}.
For the last statement we consider the space $\calCR^+( X)$ of positive crossratios 
as a closed convex cone in the topological vector space of crossratios on $ X$ with the topology of pointwise convergence.  This last property will be used in the proof of the continuity of the map which to a point in the real spectrum compactification of maximal representations associates a geodesic current \cite{BIPP-RSPEC}; see also \cite[Theorem 36]{BIPP-ann}.

The proof of Theorem \ref{thm_intro:thm6} bypasses the possible discontinuities of the crossratio $[\,\cdot\,,\,\cdot\,,\,\cdot\,,\,\cdot\,]$ 
by forcing inner and outer regularity of the current $\mu$ and using its $\sigma$-additivity. As an application of the explicit construction we obtain:
\begin{cor}\label{cor_intro:atomic}
If the crossratio $[\,\cdot\,,\,\cdot\,,\,\cdot\,,\,\cdot\,]$  is integral valued, then the current $\mu$ is a multicurve.
\end{cor}

To deduce Theorems~\ref{thm_intro:thm1} and \ref{thm_intro:multicurves}, from Theorem~\ref{thm_intro:thm6} and Corollary \ref{cor_intro:atomic},
we use the maximal framing to construct a positive crossratio $[\,\cdot\,,\,\cdot\,,\,\cdot\,,\,\cdot\,]_\rho$
on $\hg$  whose periods satisfy the equality $\per(\gamma)=L(\rho(\gamma))$.  Then Theorem~\ref{thm_intro:thm6} provides 
a geodesic current with the required properties.

Maximal representations are not the only class of representations whose length function is given by the periods of a positive crossratio: 
this is the case for all \emph{positively ratioed representations} \cite{Martone_Zhang} -- a class that also includes Hitchin representations \cite{Labourie} --, 
representations satisfying property $H_k$ \cite{BeyP1} and $\Theta$-positive representations \cite{BeyP2}. 
Corollary \ref{cor_intro:atomic} can be used to study asymptotic properties of these representations as well.

Our approach using $\sigma$-additivity of geodesic currents has interesting applications even for representations in $\PSp(2n,\bR)$, 
for which we cannot always assume that the crossratio is continuous.
The simplest instance is for $\PSL(2,\RR)$ if $\rho$ sends an element representing a cusp of $\Gamma$ to a hyperbolic element.

\begin{cor}\label{cor:cor2}  Let $\rho\colon\Gamma\to\PSp(2n,\bR)$ be a maximal representation
and let $K\subset\Sigma=\Gamma\backslash\H$ be a compact subset.
Then there are constants $0<c_1\leq c_2$ such that for every $\gamma\in\Gamma$ representing a closed geodesic $c$ 
contained in $K$
\bqn
c_1\ell(c)\leq L(\rho(\gamma))\leq c_2\ell(c)\,.
\eqn
In particular this holds uniformly for all $\gamma$ representing simple closed geodesics.
\end{cor}

This corollary is well known for Anosov representations.  
However if $\Sigma$ is not compact, a maximal representation is not necessarily Anosov 
since the images of parabolic elements can be unipotent (see for instance \S\ref{ex:strubel}).

\subsection*{Actions with compatible barycenters}
The proof of Theorem \ref{thm_intro:thm5}  is carried out in the framework of actions with compatible barycenters that we now define. 
Given an isometric $\Gamma$-action on a metric space $(\calX,d)$, we say that a map
\bqn
\beta\colon X^{(3)}\longrightarrow \calX
\eqn
from the set $ X^{(3)}$ of distinct triples in $ X$ to $\calX$ is a {\em barycenter compatible with the crossratio}
$[\,\cdot\,,\,\cdot\,,\,\cdot\,,\,\cdot\,]$ if $\beta$ is $S_3$-invariant, $\Gamma$-equivariant 
and for every $(a,b,c,d)\in X^{[4]}$, we have
\bqn
[a,b,c,d]=d(\beta(a,b,d),\beta(a,c,d))\,.
\eqn
We show then:

\begin{thm}\label{thm_intro:thm7}  Let 
$ X\subset\deH$ be a $\Gamma$-invariant non-empty subset 
and $[\,\cdot\,,\,\cdot\,,\,\cdot\,,\,\cdot\,]$ a positive crossratio on $ X$.
Assume that the geodesic current $\mu$ associated by Theorem~\ref{thm_intro:thm6}
to the positive crossratio $[\,\cdot\,,\,\cdot\,,\,\cdot\,,\,\cdot\,]$ corresponds to a measured lamination.
Then for every isometric $\Gamma$-action on a metric space $\calX$ 
admitting a barycenter compatible with the crossratio $[\,\cdot\,,\,\cdot\,,\,\cdot\,,\,\cdot\,]$, 
there is an isometric $\Gamma$-equivariant map
\bqn
\calV(\mu)\longrightarrow\calX.
\eqn
\end{thm}

We will see that Theorem \ref{thm_intro:thm7} always applies to a framed action of $\Gamma$ on an $\bR$-tree $\calT$ if the crossratio $[\,\cdot\,,\,\cdot\,,\,\cdot\,,\,\cdot\,]$ induced by the action is positive (Proposition \ref{thm: framed trees}). This crossratio is always positive in the case of the action on $\calB_1^\bF$ induced by a maximal framed representation in $\SL(2,\bF)$ (Theorem \ref{thm:SL2 bis}).

When $\rho\colon\Gamma\to\Sp(2n,\bF)$ is a maximal framed representation,  we will use the geometry of the Siegel space to define 
 a barycenter map 
associating to every maximal triple $(\ell_1,\ell_2,\ell_3)$ of Lagrangians a point $B(\ell_1,\ell_2,\ell_3)\in\calB_n^\bF$.
Given a representation $\rho\colon\Gamma\to\Sp(2n,\bF)$ with maximal framing
$\varphi\colon X\to\calL(\bF^{2n})$, we will show that the map
\bqn
\beta(a,b,c)=B(\varphi(a),\varphi(b),\varphi(c))
\eqn
defines a barycenter compatible with the crossratio $[\,\cdot\,,\,\cdot\,,\,\cdot\,,\,\cdot\,]_\rho$
previously defined. Thus Theorem \ref{thm_intro:thm7} applies whenever $\mu_\rho$ corresponds to a measured lamination. Using \cite[Corollary 1.9]{BIPP1} we can find a collection of maximal subsurfaces $\Sigma'\subset\Sigma$ such that Theorem \ref{thm_intro:thm7} holds for the restriction of $\rho$ to $\Sigma'$.

\subsection*{Structure of the paper}

In \S~\ref{sec:3} we discuss preliminaries on geodesic currents and measured laminations. The  new result is Proposition \ref{prop:4pointCriterionForLamCurr} that gives an useful 4-point characterization of measured laminations among geodesic currents that only involves a dense subset of $\partial\H$. In \S~\ref{sec:poscr} we introduce positive crossratios and the associated periods. In \S~\ref{sec:cr->gc} we construct the geodesic current associated to such a crossratio.  Theorem \ref{thm_intro:thm6} follows directly combining Propositions \ref{prop:Rm}, \ref{prop:new4.10} and \ref{prop:new4.11}, which are proven in this section. Corollary \ref{cor_intro:atomic} follows from Proposition \ref{p.atomic}. In \S~\ref{s.tree} we discuss barycenter maps, and prove Theorem  \ref{thm_intro:thm7}. In \S~\ref{sec:max_repr_van_sys} we review the geometry of the Siegel space over real closed fields from \cite{BP}, using this we associate to a maximal framed action on the $\Lambda$-metric space $\calB^\bF_n$ a positive crossratio (Proposition \ref{lem:crBP}), as well as a compatible barycenter map (\S~\ref{subsec:baryc}). In \S~\ref{sec:appl} we prove the results on maximal framed representations: Theorems \ref{thm_intro:thm1}, \ref{thm_intro:multicurves}, \ref{thm_intro:thm5} and Corollaries \ref{cor:cor4} and \ref{cor:cor2}. \S~\ref{sec:example} collects interesting examples of maximal framed representations, illustrating various phenomena.

\section{On geodesic currents and measured laminations}\label{sec:3}

In this section we recall the notions of geodesic currents and their Bonahon intersection (\S\ref{ss.Bon}); then we establish a criterion for the support of a geodesic current to be a geodesic lamination (\S\ref{s1:trees}); we end by recalling the definition of the tree $\calT(\mu)$ associated to a current $\mu$ of lamination type in terms of the straight pseudodistance (see \cite[Section 4]{BIPP1}) on $\H$ associated to a general current.  

Let $\Sigma:=\Gamma\backslash\H$ be a hyperbolic surface of finite area
and denote by $\pr\colon\H\to\Sigma$ the covering map.
The boundary $\deH$ of $\H$ is endowed with the natural cyclic order.
For $(a,b)\in(\deH)^2$ with $a\neq b$,
we will denote the associated \emph{open
  interval} in $\deH$ by
\bqn
\ioo{a}{b}:=\{x\in\deH:\,(a,x,b)\text{ is positively oriented}\}\,,
\eqn
and
the \emph{left half open interval} $\ioc{a}{b}$,  \emph{rigth
  half open interval}
$\ico{a}{b}$ and \emph{closed interval}  $\icc{a}{b}$ accordingly, 
so for example
\bqn
\ico{a}{b}=\{a\}\cup\ioo{a}{b}\,.
\eqn
Given a subset $A\subset\deH$, we will denote:
\bqn
A^{[4]}:=\{(x,y,z,t)\in A^4:\,(x,y,z,t) \text{ is positively oriented}\}\,.
\eqn
%

\subsection{Geodesic currents}\label{ss.Bon}
 A \emph{geodesic current} is a flip-invariant $\Gamma$-invariant positive Radon measure
 on the set of (oriented) geodesics in $\H$, which we identify with
\bqn
(\deH)^{(2)}:=\{(x,y)\in(\deH)^2:\,x\neq y\}\,.
\eqn
Given a (non-oriented) geodesic $c\subset\Sigma$ that is either closed or joining two
cusps,
$\delta_c$ will be the geodesic current given by 
\bqn
\delta_c:=\frac{1}{2}\sum_{(a,b)}\delta_{(a,b)}
\eqn
where we sum on the set of oriented geodesics $(a,b)\in (\deH)^{(2)}$ lifting $c$.
%

%

Two geodesics
$(a,b),(a',b')\in(\deH)^{(2)}$ are
{\em transverse} if they intersect in a point.
The group $\PSL(2,\bR)$, hence $\Gamma$, acts properly on the open subset 
$\calG \subset (\deH)^{(2)}\times(\deH)^{(2)}$
of transverse pairs of  geodesics. 
The   \emph{Bonahon intersection} $i(\mu,\nu)$ of two geodesic currents $\mu$ and $\nu$  is the  
(possibly infinite) $\mu\times\nu$-measure of any Borel fundamental
domain for $\Gamma$ in $\calG$.
Note that, when $\nu$ has compact carrier%
\footnote{ We recall that the \emph{carrier} 
of a geodesic current $\mu$ is the closed subset $\pr(\bigcup_{g\in\supp(\mu)}g)\subset\Sigma$.
The hypothesis that $\nu$ has compact carrier ensures that $i(\mu,\nu)<\infty$ and
is needed in the proofs of the continuity of the Bonahon
intersection.}%
,  $i(\mu,\nu)$  is finite. 
In order to simplify notations, we set 
\bqn
i(\mu,c):=i(\mu,\delta_c)
\eqn
for all closed geodesic $c\subset\Sigma$.
\subsection{Measured laminations and $\mu$-short geodesics}\label{s1:trees}
We refer to \cite[\S 8.3.4]{Martelli} for preliminaries on measured laminations. The equivalence between (1) and (2) in the next proposition is classical; we  establish that the two conditions are also equivalent to (4), an additional
$4$-point characterization that uses only a dense subset of $\deH$.

\begin{prop}
\label{prop:4pointCriterionForLamCurr}
Let $\mu$ be a geodesic current and 
 $X$  a dense subset of $\deH$.
The following are equivalent.

\begin{enumerate}
\item 
\label{it: supp=lamin} 
$\supp(\mu)$ is a lamination ;

\item 
\label{it: 0 autointersection}
$i(\mu,\mu)=0$ ;

\item
\label{it: oc times oc} 
$\mu(\ioc{d}{a}\times\ioc{b}{c}) \cdot \mu(\ioc{a}{b}\times\ioc{c}{d})=0$ 
for all $(a,b,c,d)$ in  $(\deH)^{[4]}$;
\item 
\label{it: oo times oo in dense X} 
 $\mu(\ioo{d}{a}\times \ioo{b}{c}) \cdot \mu(\ioo{a}{b}\times \ioo{c}{d})=0$ 
for all $(a,b,c,d)$ in  $X^{[4]}$.
\end{enumerate}
Such a current will be called {\em of lamination type}.\footnote{While in the introduction we identified with a slight abuse of notation measured laminations with currents with zero self intersection, we prefer to keep the objects distinct for the rest of the paper.}
\end{prop}

\begin{proof}
We first show (\ref{it: supp=lamin}) 
implies (\ref{it: 0 autointersection}) :
if $i(\mu,\mu)>0$ then there exist
open subsets $A,B$ of $(\deH)^{(2)}$ with 
$A\times B \subset \calG$ and 
$$(\mu\times\mu)(A\times B)=\mu(A)\mu(B)>0,$$
hence $\mu(A),\mu(B)>0$.
Then there exists  
$g\in A\cap \supp(\mu)$ and $g'\in B\cap \supp(\mu)$. 
Since  $A\times B \subset \calG$, 
in particular $g$ and $g'$ intersect in a point,
hence $\supp(\mu)$ is not a lamination.

We now prove that (\ref{it: 0 autointersection}) implies (\ref{it: oc times
  oc}) :  since $i(\mu,\mu)=0$,
 we have $(\mu\times\mu)(A\times B)=0$ for all transverse 
Borel subsets $A,B$ of $(\deH)^{(2)}$ (namely every pair of  geodesics
$(a,b)\in A$, $(a',b')\in B$ intersect in one point).
The claim follows
as $A=\ioc{d}{a}\times \ioc{b}{c}$ and $B=\ioc{a}{b}\times \ioc{c}{d}$ are
transverse.

It is clear that  (\ref{it: oc times oc}) 
implies (\ref{it: oo times oo in dense X}).
 
Suppose now (\ref{it: oo times oo in dense X}).
Let $g$, $g'$ be two geodesics in $\supp(\mu)$. 
If they are transverse, then $g=(x,y)$, $g'=(x',y')$ with
$(x,x',y,y')$ positively oriented. Then by density of $X$ 
there exists  
$(a,b,c,d)$ in $X^{[4]}$ such that
$(x,a,x',b,y,c,y',d)$ is positively oriented.
Then $g\in \ioo{d}{a}\times \ioo{b}{c}$ and $g'\in \ioo{a}{b}\times \ioo{c}{d}$, 
and as $g$, $g'$ are in
the support of $\mu$, we have 
$\mu(\ioo{d}{a}\times \ioo{b}{c})>0$ and $\mu(\ioo{a}{b}\times \ioo{c}{d})>0$, 
a contradiction.
Hence $\supp(\mu)$ is a lamination, 
proving (\ref{it: supp=lamin}).
\end{proof}
An important concept in \cite{BIPP1} was that of
\emph{$\mu$-short geodesic}, namely a  geodesic not intersecting in a point any 
geodesic in the support of $\mu$; observe that a geodesic $(a,b)$ is $\mu$-short if and only if 
\bqn
\mu(\ioo{a}{b}\times\ioo{b}{a})=0\,.
\eqn
It follows from Proposition \ref{prop:4pointCriterionForLamCurr} that if the current $\mu$ is of lamination type, its support consists of $\mu$-short geodesics.

\subsection{The tree associated to a current of lamination type}\label{s.treemu}
We now recall the construction of the tree $\calT(\mu)$ associated to a current $\mu$ of lamination type. We chose here a description adapted to our purposes, but this agrees with the standard construction described, for example, in \cite{Mor-Sha} and \cite[\S11.12]{Kapo}. 

Given a geodesic current $\mu$, we consider the straight
pseudodistance on $\H$ \cite[\S~4]{BIPP1}
\bqn
d_\mu(x,y) = \frac{1}{2} \;\big\{\mu(\calG^\pitchfork_{[x,y)}) + \mu(\calG^\pitchfork_{(x,y]})\big\}
\eqn
where for a possibly empty geodesic segment $I \subset \H$ we define
\bqn
\calG^\pitchfork_I = \{(g_-,g_+) \in (\deH)^{(2)}: |g \cap I | = 1\}
\eqn
as the set of geodesics $g$ that intersect transversely the geodesic segment $I$.

If $\mu$ is of lamination type, then the quotient metric space $X_\mu = \H/\!\!\sim$, 
  obtained by identifying points at $d_\mu$-distance zero, is $0$-hyperbolic in the sense of Gromov and 
  can therefore be canonically embedded in a minimal  $\bR$-tree $\calT(\mu)$. We will denote by $\calV(\mu)$ the image in $\calT(\mu)$ of the complementary regions $\calR$ of $\supp(\mu)$. It corresponds to the set of branching points of $\calT(\mu)$.

Since $\mu$ is $\Gamma$-invariant, the group $\Gamma$ acts on $\calT(\mu)$ and therefore on $\calV(\mu)$ by isometries. A direct consequence of the definition of Bonahon intersection is that, for this action,
$$\ell_{\calT(\mu)}(\gamma)=i(\mu,\delta_c)$$
for hyperbolic $\gamma$ representing a closed geodesic $c$.

\section{Positive crossratios}\label{sec:poscr}

In this section  we introduce the notion of
positive crossratio $[\,\cdot\,,\,\cdot\,,\,\cdot\,,\,\cdot\,]$, prove that its  periods are well defined, and discuss examples.
\subsection{Positive crossratios}\label{s.3.1}

Let $\Sigma=\Gamma\backslash\H$ be a finite area hyperbolic surface
and let $\Gamma<\PSL(2,\bR)$ be its fundamental group realized as a torsion-free lattice in $\PSL(2,\bR)$.

\begin{defn}\label{defn:cr}  Let $ X\subset\deH$ be a $\Gamma$-invariant non-empty subset.
	A {\em crossratio} on $ X$ 
	is a real valued function $[\,\cdot\,,\,\cdot\,,\,\cdot\,,\,\cdot\,]$
	defined on $ X^{[4]}$ satisfying the following properties:
	\be
	\item[(CR1)] it is $\Gamma$-invariant;
	\item[(CR2)] $[x,y,z,t]=[z,t,x,y]$ 
	for all $(x,y,z,t)\in  X^{[4]}$;
	\item[(CR3)] $[x,y,z,t]+[x,z,w,t]=[x,y,w,t]$ 
	whenever $(x,y,z,w,t)$ is positively oriented.
	\ee
	
	The crossratio is in addition {\em positive} if
	\be
	\item[(CR4)] $[x,y,z,t]\geq 0$ 
	for all $(x,y,z,t)\in X^{[4]}$.
	\ee
\end{defn}

\begin{remark}\label{rem:cr}
	There are many different non-equivalent notions of crossratio available in the literature, 
	and there is no standard choice %
	of the order of the arguments of the function
	$[\,\cdot\,,\,\cdot\,,\,\cdot\,,\,\cdot\,]$.
	More specifically \begin{itemize}
		\item if 
		$B(\cdot,\cdot,\cdot,\cdot)$ is a crossratio according to \cite[Definition 2.4]{Martone_Zhang} 
		(which agrees with \cite[Definition 1.f]{Ledrappier}), then
		\bqn
		[a,b,c,d]=B(a,d,c,b)
		\eqn
		is a crossratio according to Definition~\ref{defn:cr}. However we do not require continuity, and our crossratio is defined only on a smaller set. 
		
		\item if $Cr(\cdot,\cdot,\cdot,\cdot)$ is a crossratio according to \cite[p.~1]{H5}, then
		\bqn
		[a,b,c,d]=\log Cr(b,c,d,a)
		\eqn
		is a crossratio according to Definition~\ref{defn:cr}. 
		However, in \cite[p.~1]{H5} the crossratio is  H\"older continuous.
		
		\item if $\mathbb B(\cdot,\cdot,\cdot,\cdot)$ is a crossratio according to \cite[p.~1]{Labourie}, then
		\bqn
		[a,b,c,d]=\log \mathbb B(a,c,d,b)
		\eqn
		is a crossratio according to Definition~\ref{defn:cr}. 
		However, the definition in \cite[p.~1]{Labourie}  requires H\"older continuity and a much stronger positivity than what we impose, namely the strict inequality in (CR4).
	\end{itemize}
\end{remark}

As a direct consequence of (CR3),
positive crossratios have the following monotonicity property\smallskip

\noindent
\begin{minipage}{.7\textwidth}
	\begin{lem}
		For all $(x_1,x_2,x_3,x_4)$ and $(x'_1,x'_2,x'_3,x'_4)$  in $ X^{[4]}$
		such that 
		\bqn
		\ioo{x_4}{x_1}\subset\ioo{x'_4}{x'_1}\quad\text{ and }\quad\ioo{x_2}{x_3}\subset\ioo{x'_2}{x'_3}\,,
		\eqn
		we have
		\bq\label{eq:cr_monotone}
		[x_1,x_2,x_3,x_4]\leq[x'_1,x'_2,x'_3,x'_4]\,.
		\eq
	\end{lem}
	
\end{minipage}
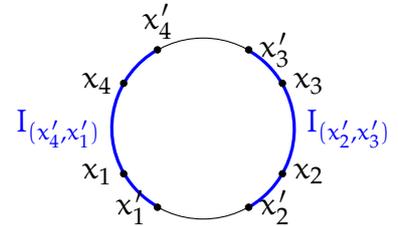
\begin{minipage}{.3\textwidth}
	\hskip.5cm
	\begin{tikzpicture}[scale=1.2]
		\draw (0,0) circle [radius=1cm];
		\draw[blue, very thick] (.5,-.87) arc[start angle=-60, end angle=60, radius=1];
		\draw[blue, very thick] (-.5,.87) arc[start angle=120, end angle=240, radius=1];
		
		\draw (-1,0) node [left,blue] { $\ioo{x'_4}{x'_1}$};
		\draw (1,0) node [right,blue] { $\ioo{x'_2}{x'_3}$};
		\filldraw (.87,.5) circle [radius=1pt] node[right] {$x_3$};
		\filldraw (.87,-.5) circle [radius=1pt] node[right] {$x_2$};
		\filldraw (-.87,.5) circle [radius=1pt] node[left] {$x_4$};
		\filldraw (-.87,-.5) circle [radius=1pt] node[left] {$x_1$};
		
		\filldraw (.5,.87) circle [radius=1pt] node[right] {$x'_3$};
		\filldraw (.5,-.87) circle [radius=1pt] node[right] {$x'_2$};
		\filldraw (-.5,.87) circle [radius=1pt] node[left, above] {$x'_4$};
		\filldraw (-.5,-.87) circle [radius=1pt] node[left] {$x'_1$};
		
	\end{tikzpicture}
\end{minipage}

\smallskip

To gain some intuition on the properties (CR2) and (CR3), we recall that if 
$x,y,z,t\in\deH=\bR\cup\{\infty\}$ and
\bqn
[x,y,z,t]=\ln\frac{(x-z)(y-t)}{(x-y)(z-t)}
\eqn
is the logarithm of the usual crossratio, 
the Liouville measure $\calL$ has the property that 
\bq\label{eq:liou_cr}
\calL(\ioo{t}{x}\times\ioo{y}{z})=[x,y,z,t]\,.
\eq

\noindent\begin{minipage}{.5\textwidth}
	Thus (CR2) corresponds to the flip-invariance of $\calL$
	\bqn
	\calL(\ioo{t}{x}\times\ioo{y}{z})=\calL(\ioo{y}{z}\times\ioo{t}{x})
	\eqn
	and (CR3) to additivity
	\bqn
	\calL(\ioo{t}{x}\times\ioo{y}{w})
	=\calL(\ioo{t}{x}\times\ioo{y}{z})+\calL(\ioo{t}{x}\times\ioo{z}{w})
	\eqn
since
	\bqn
	\calL(\ioo{t}{x}\times\{z\})=0\,.
	\eqn
\end{minipage}
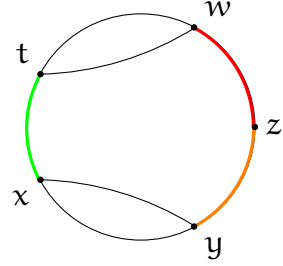
\begin{minipage}{.5\textwidth}
	\hskip2.5cm
	\begin{tikzpicture}
		\draw (0,0) circle [radius=1.5cm];
		\draw[orange, very thick] (.7,-1.32) arc[start angle=-61.5, end angle=0, radius=1.5];
		\draw[red, very thick] (1.48,0) arc[start angle=0, end angle=60, radius=1.5];
		\draw[green, very thick] (-1.32,.7) arc[start angle=151.5, end angle=208, radius=1.5];
		\filldraw (1.5,0) circle [radius=1pt] node[right] {$z$};
		\filldraw (0.7,1.32) circle [radius=1pt] node[above right] {$w$};
		\filldraw (0.7,-1.32) circle [radius=1pt] node[below right] {$y$};
		\filldraw (-1.32,-0.7) circle [radius=1pt] node[below left] {$x$};
		\filldraw (-1.32,0.7) circle [radius=1pt] node[above left] {$t$};
		\draw (-1.32,0.7) arc[start angle=271.5, end angle=302.5, radius=4];
		\draw (-1.32,-0.7) arc[start angle=88.5, end angle=58, radius=4];
	\end{tikzpicture}
\end{minipage}


\subsection{Examples}
There are two natural crossratios associated to a geodesic current:
\begin{example}\label{ex:crmu}
	If $\mu$ is a current, 
	it is easily checked that
	$$[a,b,c,d]_{\mu}^+:=\mu(\ioc{d}{a}\times\ioc{b}{c})$$
	defines a positive crossratio on $\deH$.
	Similarly,
	$$[a,b,c,d]_{\mu}^-:=\mu(\ico{d}{a}\times\ico{b}{c})$$
	defines a positive crossratio on $\deH$.
	Note that these two crossratios may be different (for example this is the case if
	$\mu=\delta_c$ for some closed geodesic $c$).
\end{example}

Framed actions on trees give other fundamental examples of crossratios:
\begin{example}\label{ex:trees} If $\calT$ is a real tree, 
	we denote by
	$[\,\cdot\,,\,\cdot\,,\,\cdot\,,\,\cdot\,]_\calT$
	the usual crossratio  on the boundary $\partial_\infty\calT$ of the
	tree $\calT$: for every pairwise distinct 
	$(a,b,c,d)\in \partial_\infty\calT^4$,
	$[a,b,c,d]_\calT$ is the signed distance, 
	on the oriented geodesic 
	from $a$ to $d$,  from the orthogonal projection $\beta_\calT(a,b,d)$ of $b$
	to the orthogonal projection $\beta_\calT(a,c,d)$ of $c$.
	Note that
	\begin{equation}
		\label{eq: comp barycenter in a tree}
		|[a,b,c,d]_\calT|=d(\beta_\calT(a,b,d),\beta_\calT(a,c,d))
	\end{equation}
	where $d$ denotes the distance in $\calT$.
	\end{example}
	%
	
	A {\em framed action} of $\Gamma$ on $\calT$ is 
	an action by isometries $\rho\colon\Gamma\to\Isom(\calT)$ admitting
	an injective equivariant map (a {\em  framing})
	$\varphi\colon X\to \partial_\infty\calT$
	where $ X$ is some $\Gamma$-invariant non-empty subset of $\deH$.
	Then the  crossratio $[\,\cdot\,,\,\cdot\,,\,\cdot\,,\,\cdot\,]_\calT$
	on $\partial_\infty\calT$ induces a crossratio
	$[\,\cdot\,,\,\cdot\,,\,\cdot\,,\,\cdot\,]_\varphi$ on $ X$ defined by
	\bqn
	[x_1,x_2,x_3,x_4]_\varphi
	:=[\varphi(x_1),\varphi(x_2),\varphi(x_3),\varphi(x_4)]_\calT
	\eqn
	for every $(x_1,x_2,x_3,x_4)\in X^{[4]}$.

\begin{example}
An example of such situation is given by the $\Gamma$-action on the $\bR$-tree $\calT(\mu)$
associated to a current of lamination type.  Let $X\subset\deH$ be the set of fixed point of hyperbolic elements 
whose axis are transverse to the geodesic lamination $\supp(\mu)$.  Then for every such $\gamma\in\Gamma$
with $\{\gamma_-,\gamma_+\}\subset X$, the element $\gamma$ acts on $\calT(\mu)$ with strictly positive translation length 
$\ell_{\calT(\mu)}(\gamma)=i(\mu,\delta_c)$ (see \S~\ref{s.treemu}) and has thus an attractive fixed point $\varphi(\gamma_+)$
and a repulsive one $\varphi(\gamma_-)$ in $\partial_\infty\calT(\mu)$.
Then $\varphi\colon X\to\partial_\infty\calT(\mu)$ is a framing and it follows from the definition of the distance on $\calT(\mu)$ that
\bqn
\mu(\icc{x_4}{x_1}\times\icc{x_2}{x_3})=[\varphi(x_1), \varphi(x_2), \varphi(x_3), \varphi(x_4)]_\calT\,.
\eqn
It follows from the discussion recalled in \S\ref{s1:trees} that the crossratio is positive.
\end{example}

Example \ref{ex:trees} inspires the following definition:
\begin{defn}
	\label{def: ultrametric CR}
	We say that a crossratio is {\em ultrametric} if it satisfies :
	\begin{center}
		(CRU) $[a,b,c,d] \times [b,c,d,a]=0$ 
		for all $(a,b,c,d)$ in  $ X^{[4]}$.
	\end{center}
\end{defn}

The following is clear.
\begin{prop}
	\label{prop: cr from tree is ultram}
	The crossratio induced by a framed action on a $\bR$-tree 
	is ultrametric.
\end{prop}

The following is a corollary of 
Proposition \ref{prop:4pointCriterionForLamCurr}.
\begin{prop}
	\label{prop: cr from lamin curr is ultram}
	The crossratio 
	$$[a,b,c,d]:=\mu(\ioc{d}{a}\times\ioc{b}{c})$$
	associated to a lamination type current
	$\mu$ is ultrametric.
\end{prop}

\subsection{The periods of the crossratio}\label{ssec:periods}
Let now $[\,\cdot\,,\,\cdot\,,\,\cdot\,,\,\cdot\,]$ be a positive crossratio defined on $ X\subset\deH$,
and let $\gamma\in\Gamma$  be hyperbolic such that $\{\gamma_-,\gamma_+\}\subset X$.
The additivity property (CR3) of the crossratio implies that the value 
$[\gamma_-,x,\gamma x,\gamma_+]$ is independent of $x\in\ioo{\gamma_-}{\gamma_+}\cap X$.
This justifies the following:

\begin{defn}\label{d.period} If $[\,\cdot\,,\,\cdot\,,\,\cdot\,,\,\cdot\,]$ is a positive crossratio on $ X$
	and $\gamma\in\Gamma$ is hyperbolic such that $\{\gamma_-,\gamma_+\}\subset X$,
	the {\em period} of $\gamma$ is defined as
	\bqn
	\per(\gamma):=[\gamma_-,x,\gamma x,\gamma_+]\,.
	\eqn
        for one (any) $x\in\ioo{\gamma_-}{\gamma_+}\cap X$.
\end{defn}

The purpose of this section is to extend the definition of the period of a crossratio defined on a $\Gamma$-invariant set $ X\subset \partial\H$ to hyperbolic elements $\g\in\Gamma$ whose endpoints do not necessarily belong to the set $ X$. 

This is achieved by the following:
\begin{prop}\label{p.welldef}
	Let $[\,\cdot\,,\,\cdot\,,\,\cdot\,,\,\cdot\,]$  be a positive crossratio on $ X$, and $\g\in\Gamma$ be a hyperbolic element. Choose monotone sequences $(x_n), (x_n')\subset X$ with limit $\gamma_-$ and $(y_n), (y_n')\subset X$ with limit $\gamma_+$. Assume furthermore that $(x_n',\gamma_-,x_n,y_n,\g_+,y_n')$ is positively oriented. Then, for all $x\in X$,
	$$\lim_{n\to\infty}[x_n,x,\g x, y_n]=\lim_{n\to\infty}[x_n',x,\g x, y_n'].$$ 
\end{prop}
\begin{figure}[h]
	\begin{center}
		\begin{tikzpicture}
			\draw (0,0) circle [radius=1.5cm];
			\filldraw (-1.0605, 1.0605) circle [radius=1pt] node[above left] {$x_n'$};
			\filldraw (-1.2975,.75) circle [radius=1pt] node[left] {$\gamma_-$};
			\filldraw (-1.5, 0) circle [radius=1pt] node[left] {$x_n$};
			
			\filldraw (1.0605, 1.0605) circle [radius=1pt] node[above right] {$y_n'$};
			\filldraw (1.2975,.75) circle [radius=1pt] node[right] {$\gamma_+$};
			\filldraw (1.5, 0) circle [radius=1pt] node[right] {$y_n$};
			\draw [red](-1.2975,.75)  to [out= -20 , in= 200] (1.2975,.75);
			
			\filldraw (-.75,-1.2975) circle [radius=1pt] node[below left] {$x$};
			\filldraw (.75,-1.2975) circle [radius=1pt] node[below right] {$\gamma x$};
		\end{tikzpicture}
	\end{center}
	\caption{Proposition~\ref{p.welldef}}\label{f.periods}
\end{figure}
\begin{proof}
	Up to passing to a subsequence we can and will assume that $(x_0,x,\gamma x, y_0)$ is positive (see Figure \ref{f.periods}).
	Since by (CR2) and (CR3) we have
	$$[x_n,x,\g x, y_n]=[x_n,x,\g x, x_n']+[x_n',x,\g x, y_n']+[y_n',x,\g x, y_n],$$
	it is enough to show that 
	$$\lim_{n\to\infty}[x_n,x,\g x, x_n']=0$$
	and the analogous statement for $[y_n',x,\g x, y_n]$.
	
	Since, by (CR4), the crossratio is positive, and $\g^{-n} x_0\to \gamma_-$, it is in turn enough to show that 
	$$\lim_{n\to\infty}[\g^{-n} x_0,x,\g x,\g^{-n} x_0']=0.$$
	This follows since, for every $N$,
	\bqn
	\ba[]
	\infty>[x_0,x,y_0',x_0']
	&\geq [x_0,x,\g^N x,x_0'] \\
	&\geq \sum_{j=0}^{N-1}[x_0,\g^jx,\g^{j+1} x,x_0'] \\
	&= \sum_{j=0}^{N-1}[\g^{-j}x_0,x,\g x,\g^{-j}x_0']. \\
	\ea
	\eqn
	Here in the last equality we used that the crossratio is $\Gamma$-invariant. The claim for $[y_n',x,\g x, y_n]$ follows analogously.
\end{proof}
Thanks to Proposition \ref{p.welldef} we can extend Definition \ref{d.period} to
\begin{defn}\label{d.period2} If $[\,\cdot\,,\,\cdot\,,\,\cdot\,,\,\cdot\,]$ is a positive crossratio on $ X$
	and $\gamma\in\Gamma$ is hyperbolic,
	the {\em period} of $\gamma$ is 
	\bqn
	\per(\gamma):=\lim_{\substack{s,t \in  X\\s\to \gamma_-\\ t\to\gamma_+}}[s,x,\gamma x,t]
	\eqn
		for one  (any)
	$x\in\ioo{\gamma_-}{\gamma_+}\cap X$.
      \end{defn}

\section{The geodesic current associated to a positive crossratio}\label{sec:cr->gc}
In this section we prove Theorem \ref{thm_intro:thm6} and Corollary
\ref{cor_intro:atomic}. The proof of Theorem \ref{thm_intro:thm6} is
carried out in three steps:  in
\S\ref{s.3.2} we use a crossratio to construct a geodesic current
$\mu_{[\,\cdot\,,\,\cdot\,,\,\cdot\,,\,\cdot\,]}$; in \S\ref{s.3.3}
we relate the periods of the crossratio
$[\,\cdot\,,\,\cdot\,,\,\cdot\,,\,\cdot\,]$  and the intersection of
curves with $\mu_{[\,\cdot\,,\,\cdot\,,\,\cdot\,,\,\cdot\,]}$; in
\S\ref{s.3.4} we conclude the proof of Theorem \ref{thm_intro:thm6} by
showing that $\mu_{[\,\cdot\,,\,\cdot\,,\,\cdot\,,\,\cdot\,]}$ depends
continuously on the crossratio $[\,\cdot\,,\,\cdot\,,\,\cdot\,,\,\cdot\,]$. 
The fact that an integer valued crossratio leads to a multicurve (Corollary \ref{cor_intro:atomic})
is shown in \S\ref{s.3.5}. 
We conclude the section discussing in \S\ref{s.rest} how crossratios 
and geodesic currents can be restricted to subsurfaces; 
this is for future reference and will be  used in the study of  the real spectrum compactification of maximal representations.

\subsection{Construction of the current}\label{s.3.2}
The aim of this section is to show that
a positive crossratio always leads to a geodesic current.
 This is done in Proposition \ref{prop:Rm}.
The strategy of the proof is 
first to associate to the crossratio a
finitely additive set function $\cRo$  
defined on the family of proper rectangles with vertices 
in $ X^{[4]}$
(Proposition \ref{prop: finite add mesure on Rec}),
 and then to build a canonical Radon measure
$\mu$ out of  $\cRo$. 

Fix a $\Gamma$-invariant 
non-empty subset $ X\subset\deH$,
and a positive crossratio 
$$[\,\cdot\,,\,\cdot\,,\,\cdot\,,\,\cdot\,]\colon X^{[4]}\longrightarrow[0,\infty).$$

A {\em rectangle} in $(\deH)^{(2)}$ is the product $R=I\times J$ of
two disjoint intervals $I,J\subset \deH$.
It is called {\em proper} 
if its closure in $(\deH)^{(2)}$ is compact,
that is $\overline I\cap\overline J=\varnothing$, and $I$ and $J$ have non-empty interior.
The {\em vertices} of $R$ are then the unique
positively oriented $4$-tuple $(a,b,c,d)$ in  $(\deH)^{[4]}$
such that $d,a$ are the enpoints of $I$ and $b,c$ are
the endpoints of $J$, equivalently
$$\ioo{d}{a}\times \ioo{b}{c}\subset R \subset \icc{d}{a}\times
\icc{b}{c}\; .$$
For $\calA\subset \deH$, we denote  $\Rec(\calA)$ the family 
of all proper rectangles with vertices in 
   $\calA^{[4]}$.

If $R$ is a proper rectangle with vertices $(a,b,c,d)$  in
  $ X^{[4]}$ we define
$$\cR{R}= [a,b,c,d]$$
the {\em crossratio} of the rectangle $R$.
It follows directly from the additivity property (CR3) of the crossratio that this defines a finitely
additive positive  function on the family $\Rec( X)$
of all proper rectangles with vertices 
in  $ X^{[4]}$:

\begin{prop}
\label{prop: finite add mesure on Rec} 
The function $\cRo\colon\Rec( X)\to\bR$ satisfies the following.
\begin{enumerate}
\item If
a rectangle  $R\in\Rec( X)$ is the  union $R=R_1\sqcup R_2$ of two rectangles with disjoint interior in $\Rec( X)$, then
\bqn
\cR{R}=\cR{R_1}+\cR{R_2}.
\eqn

\item For all $R,R'$
in $\Rec( X)$,
if $R\subset R'$ then $\cR{R}\leq \cR{R'}$ .

\end{enumerate}
\end{prop}


\begin{remark}
The function $\cRo\colon\Rec( X) \to \bR$ 
may  not be $\sigma$-additive, even restricting to the family
of left half open rectangles 
$\ioc{d}{a}\times \ioc{b}{c}$  
 with $(a,b,c,d)\in  X^{[4]}$.
For example setting 
$$[a,b,c,d]:=\delta_e(\ico{d}{a}\times \ico{b}{c})$$
 for a closed curve $e$ corresponding to some hyperbolic $\gamma\in\Gamma$,
we get a positive crossratio
$[\,\cdot\,,\,\cdot\,,\,\cdot\,,\,\cdot\,]$ 
on $ X^{[4]}=\deH$ whose associated function
$\cRo$ is not $\sigma$-additive on $\Rec( X)$.
Take $a$,$c$ such that 
$(a,\gamma_+,c,\gamma_-)\in  X^{[4]}$ and
$\icc{\gamma_-}{a}\times \icc{\gamma_+}{c}$ contains no other point of
the orbit of
$(\gamma_-,\gamma_+)$.
Let $d_n\downarrow \gamma_-$ in $\ioo{\gamma_-}{a}$ 
and
$b_n\downarrow \gamma_+$ in $\ioo{\gamma_+}{c}$.
Let $R_n=\ioc{d_n}{a}\times \ioc{b_n}{c}$. Then
$R=\ioc{\gamma_-}{a}\times \ioc{\gamma_+}{c}$ is 
the increasing union of the $R_n$, and 
$\cR{R_n}=[a,b_n,c,d_n]=0$ for all $n$ whereas 
$\cR{R}=[a,\gamma_+,c,\gamma_-]=1$, contradicting
$\sigma$-additivity.
The problem is due to the fact that this crossratio is not
continuous at $(a,\gamma_+,c,\gamma_-)$.
\end{remark}

We now construct the mesure $\mu$.
Recall that for a rectangle $R\in\Rec( X)$
  with vertices $(a,b,c,d)$ we set
$\cR{R}= [a,b,c,d]$.
Furthermore we denote by $\ovcirc{R}$ (resp. $\ov{R}$) the open (resp. closed) rectangle with the same vertices as $R$.
\begin{prop}\label{prop:Rm} 
There exists a unique positive Radon measure $\mu$ on $(\deH)^{(2)}$
satisfying 
one of the following equivalent conditions.
\begin{enumerate}
\item 
\label{it: mu(int) leq cr leq mu(closure)}
 $\mu(\ovcirc R)\leq\cR{R}\leq \mu(\ov{R})$ for any (proper) rectangle $R\in\Rec( X)$.

\item
\label{it: R'<R <R''}
 For any (proper) rectangles $R, R'\in\Rec( X)$
 with $\ov{R'}\subset \ovcirc{R}$, we have  
 $\mu(R')\leq \cR{R}$  
 and $\cR{R'}\leq \mu(R)$.

\item
\label{it: open R}
For all  (proper) open rectangles $R\in\Rec(\deH)$
 $$\mu(R)=
 \sup\left\{ \cR{R'} :\, 
 R'\in \Rec( X) \text{ and }  
 \ov{R'} \subset R \right\}.
 $$

\item 
\label{it: closed R}
For all (proper) closed rectangles $R\in\Rec(\deH)$
 $$\mu(R)=
 \inf\left\{ \cR{R'}:\, 
 R'\in \Rec( X) \text{ and }  
 R \subset \ovcirc{R}' \right\}.$$

\end{enumerate}

\end{prop}

\begin{defn}\label{defn:gcCR}  We call the measure $\mu$ in Proposition~\ref{prop:Rm} the 
{\em geodesic current associated to the positive crossratio}
$[\,\cdot\,,\,\cdot\,,\,\cdot\,,\,\cdot\,]\colon X^{[4]}\to\bR$,
and denote it by $\mu_{[\,\cdot\,,\,\cdot\,,\,\cdot\,,\,\cdot\,]}$ if we want to emphasize the dependence 
on $[\,\cdot\,,\,\cdot\,,\,\cdot\,,\,\cdot\,]$.
\end{defn}

Proposition \ref{prop:Rm} implies the following ``outer and inner'' 
continuity properties of the current.
\begin{prop}
\label{lem:2.5}  
Let $(a,b,c,d)$ be a positively oriented quadruple in $\deH$.
Let $(a_n, b_n, c_n, d_n)_{n\geq1}$ be a sequence in $ X^{[4]}$
converging to $(a,b,c,d)$.
Then 
\begin{enumerate}
\item 
If $d_n,a_n\in\ioo{d}{a}$ and $b_n,c_n\in\ioo{b}{c}$ for all
  $n\geq1$, \\then 
$\mu(\ioo{d}{a}\times\ioo{b}{c})=\lim_n[a_n,b_n,c_n,d_n]$;

\item If $a_n,b_n\in\ioo{a}{b}$ and $c_n,d_n\in\ioo{c}{d}$ for all
  $n\geq1$, \\then 
$\mu(\icc{d}{a}\times\icc{b}{c})=\lim_n[a_n,b_n,c_n,d_n]$.
\end{enumerate}






\end{prop}

\begin{proof}  We prove the first assertion (the second is similar).
We have by (3)
\begin{equation}\label{e.rect}
	\mu(\ioo{d}{a}\times\ioo{b}{c})=
\sup\left\{[a',b',c',d']\left|\, \begin{array}{l}
a',d'\in\ioo{d}{a},\; \\b',c'\in\ioo{b}{c} 
\text{ and } \\(a',b',c',d')\in X^{[4]}\end{array}\right.\right\}\,,
\end{equation}
Let $(a',b',c',d')\in\calX^{[4]}$ 
 with $a',d'\in\ioo{d}{a}$
and $b',c'\in\ioo{b}{c}$.  For $n$ large enough we have
\bqn
\icc{d'}{a'}\subset\icc{d_n}{a_n}\qquad\text{ and }\qquad\icc{b'}{c'}\subset\icc{b_n}{c_n}\,,
\eqn 
and hence
\bqn
[a',b',c',d']\leq[a_n,b_n,c_n,d_n]\,,
\eqn
which by \eqref{e.rect} implies (1).
\end{proof}


\begin{proof}[Proof of Proposition~\ref{prop:Rm}]  
  We  begin by proving that the conditions are equivalent.
Let $\mu$ be any positive Radon measure on $(\deH)^{(2)}$.

It is clear that 
(\ref{it: mu(int) leq cr leq mu(closure)})
implies (\ref{it: R'<R <R''}).
%

We now prove that 
(\ref{it: R'<R <R''}) implies 
(\ref{it: open R}).
Consider a open rectangle $R$ in $\Rec(\deH)$.
First observe that for every $R'$ 
in $\Rec( X)$ such that $\ov{R'} \subset R$, we have 
by (\ref{it: R'<R <R''}) that $\cR{R'}\leq \mu(R)$, hence
$$ \sup\left\{\cR{R'}:\, 
R'\in \Rec( X)
\text{ and }  \ov{R'} \subset R \right\} \leq \mu(R)\;.$$
By density of $ X$ in $\deH$
we can now take an increasing sequence of rectangles $R_n$ 
in $\Rec( X)$ with union $R$ 
such that $\ov{R_n}\subset \ovcirc{R}_{n+1}$.
We have by (\ref{it: R'<R <R''}) that
 $$\mu(R_{n-1})\leq \cR{R_n}\leq \mu(R_{n+1}) $$
in particular by  $\sigma$-additivity of $\mu$ we have 
$\mu(R)=\lim_n \mu(R_n)=\lim_n \cR{R_n}$.

We now prove  that 
(\ref{it: open R})
implies
(\ref{it: closed R}).
Consider  a proper closed rectangle 
$R$ in $\deH$.
For every $R''$ 
in $\Rec(X)$ such that $R \subset \ovcirc{R}''$, 
there is an open $R'$ in $\Rec(X)$ such that $R\subset R'$ and $\ov{R'}\subset
\ovcirc{R}''$. Then we have  
$\mu(R)\leq \mu(R')\leq \cR{R''}$ by (\ref{it: open R}). 
Hence $$ \mu(R)\leq  \inf\left\{\cR{R''}:\, 
R''\in \Rec( X)
\text{ and } R \subset \ovcirc{R}'' \right\} \;.$$
Let now $R_n$ be a decreasing sequence of open rectangles 
in $\Rec(X)$ with intersection $R$,
such that $\ov{R_{n+1}}\subset R_{n}$.
Then $\mu(R)=\lim_n \mu(R_n)$, and
by (\ref{it: open R}) we have $\cR{R_n}\leq \mu(R_{n-1})$
and 
$$\mu(R_n)=\sup\left\{\cR{R'}:\, R'\in \Rec( X)\text{ and }
  \ov{R'} \subset R_n \right\} \leq \cR{R_n}$$
hence $\mu(R)=\lim_n \cR{R_n}$.

We finally check that (\ref{it: closed R})
implies
(\ref{it: mu(int) leq cr leq mu(closure)}).
Consider  any rectangle 
$R$ in $\Rec( X)$.
As $\cR{R}\leq \cR{R'}$ for all $R'$ containing $R$, 
taking infimum on $R'$ containing $\ov{R}$ in their interior 
we get by (\ref{it: closed R}) that
$\cR{R}\leq \mu(\ov{R})$.
Now write the open rectangle $\ovcirc{R}$ as a  increasing union
$\ovcirc{R}=\cup\uparrow R_n$ of closed rectangles $R_{n}$ in $\Rec( X)$
with 
$R_{n}\subset \ovcirc{R}_{n+1}$. Then by
(\ref{it: closed R}) we have
$\mu(R_n) \leq \cR{R_{n+1}} \leq \cR{R}$.
As $\mu(R_n)\to \mu(\ovcirc{R})$ by $\sigma$-additivity, 
we deduce $\mu(\ovcirc{R})\leq \cR{R}$.


 We now prove the existence of $\mu$ satisfying  (\ref{it: R'<R <R''}).
The strategy of the construction of $\mu$ is to use the finitely
additive function $\cRo$ 
to define the integral of compactly supported continuous functions.
This leads by the Riesz representation theorem to a Radon measure
$\mu$.

A {\em simple function}  is a linear combination $g=\sum_{i=1}^n\alpha_i\chi_{R_i}$ 
of characteristic functions of rectangles $R_i$ in $\Rec( X)$.
Define $E(g)$  by 
\bqn
E(g):=\sum_{i=1}^n\alpha_i\cR{R_i}\,.
\eqn
The additivity property of $\cRo$ on $\Rec( X)$ 
(Proposition \ref{prop: finite add mesure on Rec})
shows that 
$E(g)$ is independent of the representation of $g$ 
as linear combination of characteristic functions of proper rectangles
in $\Rec( X)$.
%
%
It implies that if $g_1$ and $g_2$ are simple functions, then
\bqn
E(g_1+g_2)=E(g_1)+E(g_2)\,.
\eqn
This property also  shows that if $g_1,g_2$ are simple and $g_1\leq g_2$, then $E(g_1)\leq E(g_2)$.
%

If now $f\geq0$ 
is a continuous function on $(\deH)^{(2)}$ with compact support, define
\bqn
I_+(f):=\sup\{E(g):0\leq g\leq f,\,g\text{ is simple}\}
\eqn
and 
\bqn
I_-(f):=\inf\{E(g): f\leq g,\,g\text{ is simple}\} \;.
\eqn
Then by uniform continuity of $f$ and density of $ X$,
we have
$I_-(f)=I_+(f)=:I(f)$.
The additivity of $I$  on positive 
continuous functions with compact
support
then follows from the fact 
that $I_-$ is super-additive, and $I_+$ is subadditive.
Then $I$ extends  to all continuous functions with compact support as a positive linear functional on the space of continuous functions with compact support,
hence corresponds to a Radon measure $\mu$.

We now prove that $\mu$ satisfies (\ref{it: R'<R <R''}).
Let 
$R$, $R'$ be rectangles  
with $\ov{R'}\subset \ovcirc{R}$. 
As there is 
a continuous function $f$ with compact support such that
$\chi_{R'}\leq f\leq \chi_{R}$, 
%
we have 
$\mu(R')\leq I(f)\leq E(\chi_{R})= \cR{R}$
whenever $R\in\Rec( X)$,
and 
$\cR{R'}=E(\chi_{R'})\leq I(f)\leq \mu(R)$
whenever $R'\in\Rec( X)$.

Uniqueness comes  from
(\ref{it: open R}),
as the class of proper open rectangles in $(\deH)^{(2)}$
is stable under finite intersection and generates the Borel
$\sigma$-algebra.
%
%
\end{proof}

\begin{remark}
	It follows from Proposition \ref{prop:Rm} (\ref{it: closed R})
	that also for all pencils $P=\{a\}\times\icc b c$,  
$$\mu(P)=
\inf\left\{ \cR{R'}:\, 
R'\in \Rec( X) \text{ and }  
P \subset \ovcirc{R}' \right\}.$$
\end{remark}

\begin{prop}
\label{prop: ultram cr give lamin}
If a positive crossratio is ultrametric, then
the associated current $\mu$ is of lamination type.
\end{prop}

\begin{proof}
By Proposition \ref{prop:4pointCriterionForLamCurr}  it is enough to
prove that, for all $(a,b,c,d)\in { X}^{[4]}$,
 $\mu(\ioo{d}{a}\times \ioo{b}{c}) \cdot \mu(\ioo{a}{b}\times \ioo{c}{d})=0$.

Let $(a,b,c,d)$ in  ${ X}^{[4]}$.
As the crossratio is ultrametric, we have
either $[a,b,c,d]=0$ or $[b,c,d,a]=0$.
As  $\mu(\ioo{d}{a}\times\ioo{b}{c})\leq [a,b,c,d]$ 
and $\mu(\ioo{a}{b}\times\ioo{c}{d})\leq [b,c,d,a]$
(see Proposition \ref{prop:Rm}(1)),
this implies that $\mu(\ioo{d}{a}\times\ioo{b}{c})=0$ or
$\mu(\ioo{a}{b}\times\ioo{c}{d})=0$.
\end{proof}

\subsection{Periods and intersections}\label{s.3.3}
We now turn to the problem of identifying the periods of a positive crossratio
with the intersections of the corresponding current.
Let then $[\,\cdot\,,\,\cdot\,,\,\cdot\,,\,\cdot\,]$ be a positive crossratio defined on $ X$,
and let $\gamma\in\Gamma$  be hyperbolic. 
Recall from \S\ref{ssec:periods} the definition of the period of $\g$ with respect to the crossratio $[\,\cdot\,,\,\cdot\,,\,\cdot\,,\,\cdot\,]$.

In the following proposition we will use the well known fact that if $\gamma$ is a hyperbolic element representing a closed geodesic $c\subset \Sigma$,  and $\mu$ is a geodesic current, since $\ioo{\gamma_+}{\gamma_-}\times\ioc{x}{\gamma x}$ is a Borel fundamental domain for the $\langle\gamma\rangle$-action on $\ioo{\gamma_+}{\gamma_-}\times\ioo{\gamma_-}{\gamma_+}$,
the intersection $i(\mu,\delta_c)$ can be computed as
\bqn
i(\mu,\delta_c)=\mu(\ioo{\gamma_+}{\gamma_-}\times\ioc{x}{\gamma x})\,.
\eqn

\begin{prop}\label{prop:new4.10} Let $\mu$ be the geodesic current associated to a positive crossratio $[\,\cdot\,,\,\cdot\,,\,\cdot\,,\,\cdot\,]$
and let $c$ be a closed geodesic represented by an hyperbolic element $\gamma\in\Gamma$. 
Then
\bqn
\per(\gamma)=i(\mu,\delta_c)\,.
\eqn
\end{prop}

\begin{proof}  In the notation of Proposition~\ref{p.welldef} (see also Figure~\ref{f.periods}) we have that 
\bq\label{eq:4.8.1}
\per(\gamma)=\lim_{n\to\infty}[x_n,x,\gamma x, y_n]
\eq
and also 
\bq\label{eq:4.8.2}
\per(\gamma)=\lim_{n\to\infty}[x'_n,x,\gamma x, y'_n]\,,
\eq
where $x\in\ioo{\gamma_-}{\gamma_+}\cap X$ is arbitrary.

For any $x\in\ioo{\gamma_-}{\gamma_+}\cap X$ and $n\in\bN$    we have
\bqn
\ba
i(\mu,\delta_c)
&=\mu(\ioo{\gamma_+}{\gamma_-}\times\ioc{x}{\gamma x})\\
&=\mu(\ioo{\gamma_+}{\gamma_-}\times\ioo{x}{\gamma x})+\mu(\ioo{\gamma_+}{\gamma_-}\times\{\gamma x\})\\
&\leq[x_n,x,\gamma x,y_n]+\mu(\ioo{\gamma_+}{\gamma_-}\times\{x\})\,,
\ea
\eqn
where the inequality follows from
Proposition~\ref{prop:Rm} (1).  By \eqref{eq:4.8.1} this implies that
\bqn
i(\mu,\delta_c)\leq\per(\gamma)+\mu(\ioo{\gamma_-}{\gamma_+}\times\{x\})\,.
\eqn

Next we have:
\bqn
i(\mu,\delta_c)
=\mu(\ioo{\gamma_+}{\gamma_-}\times\ioc{x}{\gamma x})
=\mu(\ioo{\gamma_+}{\gamma_-}\times\icc{x}{\gamma x})-\mu(\ioo{\gamma_+}{\gamma_-}\times\{x\})\,.
\eqn
Using Lemma~\ref{lem:new4.9} below this equals
\bqn
\mu(\icc{\gamma_+}{\gamma_-}\times\icc{x}{\gamma x})-\mu(\ioo{\gamma_+}{\gamma_-}\times\{x\})\,,
\eqn
which, again by Proposition~\ref{prop:Rm} (1), implies
\bqn
\ba
i(\mu,\delta_c)
=&\mu(\icc{\gamma_+}{\gamma_-}\times\icc{x}{\gamma x})-\mu(\ioo{\gamma_+}{\gamma_-}\times\{x\})\\
\geq&[x'_n,x,\gamma x,y'_n]-\mu(\ioo{\gamma_+}{\gamma_-}\times\{x\})\,.
\ea
\eqn
Then it follows from  \eqref{eq:4.8.2} that
\bqn
i(\mu,\delta_c)\geq\per(\gamma)-\mu(\ioo{\gamma_-}{\gamma_+}\times\{x\})
\eqn
and thus
\bqn
|i(\mu,\delta_c)-\per(\gamma)|\leq\mu(\ioo{\gamma_+}{\gamma_-}\times\{x\})
\eqn
for any $x\in\ioo{\gamma_-}{\gamma_+}\cap X$.  

Now fix a closed interval $\icc{a}{b}\subset\ioo{\gamma_-}{\gamma_+}$
with non-empty interior.  Then
\bqn
\sum_{x\in X\cap[a,b]}\mu(\ioo{\gamma_+}{\gamma_-}\times\{x\})
\leq\mu(\ioo{\gamma_+}{\gamma_-}\times\icc{a}{b})
<\infty
\eqn
and since $ X\cap[a,b]$ is infinite, this implies the existence 
of a sequence $(x_n)_{n\geq1}$ in $ X\cap[a,b]$ with
\bqn
\lim_{n\to\infty}\mu(\ioo{\gamma_+}{\gamma_-}\times\{x_n\})=0\,,
\eqn
which implies that $i(\mu,\delta_c)=\per(\gamma)$.
\end{proof}


\subsection{The current depends continuously on the crossratio}\label{s.3.4}

The vector space $\calCR( X)$ of crossratios on $ X$
is a topological vector space
for the topology of pointwise convergence and the space $\calCR^+( X)$
of positive crossratios is a closed convex cone in it.
We observe moreover that the map
\bqn
\ba
\calCR^+( X)&\longrightarrow\Curr(\Sigma)\\
[\,\cdot\,,\,\cdot\,,\,\cdot\,,\,\cdot\,]&\mapsto\mu_{[\,\cdot\,,\,\cdot\,,\,\cdot\,,\,\cdot\,]}
\ea
\eqn
from positive crossratios to the space $\Curr(\Sigma)$ of geodesic currents is surjective.  
In fact, if $\mu$ is a geodesic current,  one verifies using the regularity of $\mu$ that
$\mu_{[\,\cdot\,,\,\cdot\,,\,\cdot\,,\,\cdot\,]_{\mu}^+}=\mu$ for the crossratio
$[\,\cdot\,,\,\cdot\,,\,\cdot\,,\,\cdot\,]_{\mu}^+$
of Example \ref{ex:crmu}. 

{Let $\hg$ denote the subset of $\deH$ consisting of the fixed points of hyperbolic elements in $\Gamma$.
	Recall that if $\gamma\in\Gamma$ is hyperbolic we let $\gamma_+$ and $\gamma_-$ denote respectively 
	the attractive and the repulsive fixed point of $\gamma$.
	For every $a\in\hg$, we choose $\gamma$ such that $a=\gamma_-$ and we denote by $\overline{a}$ the point $\gamma_+$.
	
	The following simple lemma is crucial:
	
	\begin{lem}[{\cite[Proposition 8.2.8]{Martelli}}]\label{lem:new4.9}  Let $\mu$ be a geodesic current, $a\in\hg$ 
		and $I\subset\deH$ be a closed interval with $\{a,\overline{a}\}\cap I=\varnothing$.
		Then $\mu(\{a\}\times I)=0$.
	\end{lem}
	%
	%
	%
}
We call a $\Gamma$-invariant subset $S_\Gamma\subset\hg$ {\em symmetric}
if $\overline\xi\in S_\Gamma$ whenever $\xi\in S_\Gamma$.

\begin{prop}\label{prop:new4.11} Let $S_\Gamma\subset\hg$ be a $\Gamma$-invariant symmetric subset 
such that $\{(\xi,\overline{\xi}):\xi\in S_\Gamma\}$ is dense in $(\deH)^2$.  Then the map
\bqn
\ba
\calCR^+(S_\Gamma)\,&\longrightarrow\Curr(\Sigma)\\
[\,\cdot\,,\,\cdot\,,\,\cdot\,,\,\cdot\,]&\mapsto\mu_{[\,\cdot\,,\,\cdot\,,\,\cdot\,,\,\cdot\,]}
\ea
\eqn
is continuous.
\end{prop}

Observe that Proposition \ref{prop:new4.11} applies to $S_\Gamma=\hg$ in particular.
In the proof of Proposition  \ref{prop:new4.11} we will focus on a special subset of $S_\Gamma^{[4]}$: we say that
a quadruple $(a,b,c,d)\in S_\Gamma^{[4]}$ is {\em good} 
if $\{\overline{a},\overline{d}\}\cap \icc{b}{c}=\varnothing$ and 
$\{\overline{b},\overline{c}\}\cap \icc{d}{a}=\varnothing$.

\begin{lem}\label{lem:new4.12}  Let $[\,\cdot\,,\,\cdot\,,\,\cdot\,,\,\cdot\,]\in\calCR^+(S_\Gamma)$
and let $\mu$ be the associated current.  For every good quadruple $(a,b,c,d)\in S_\Gamma^{[4]}$ we have
\bqn
\mu(\ioo{d}{a}\times\ioo{b}{c})
=\mu(\icc{d}{a}\times\icc{b}{c})
=[a,b,c,d]\,.
\eqn
\end{lem}

\begin{proof}  This follows immediately from the inequalities in
  Proposition \ref{prop:Rm} (\ref{it: mu(int) leq cr leq mu(closure)})
and Lemma~\ref{lem:new4.9} applied to $\{d\}\times\icc{b}{c}$, 
$\{a\}\times\icc{b}{c}$, $\icc{d}{a}\times\{b\}$ and $\icc{d}{a}\times\{c\}$.
\end{proof}

\begin{proof}[Proof of Proposition~\ref{prop:new4.11}]  Let $[\,\cdot\,,\,\cdot\,,\,\cdot\,,\,\cdot\,]_n$ and $[\,\cdot\,,\,\cdot\,,\,\cdot\,,\,\cdot\,]$
be positive crossratios on $S_\Gamma$ and $\mu_n$, $\mu$ the associated geodesic currents.  
Assume that 
\bqn
\lim_n[\,\cdot\,,\,\cdot\,,\,\cdot\,,\,\cdot\,]_n=[\,\cdot\,,\,\cdot\,,\,\cdot\,,\,\cdot\,]\,.
\eqn
We have to show that for every positive continuous function $f$ on $(\deH)^{(2)}$ with compact support,
$\lim_{n\to\infty}\mu_n(f)=\mu(f)$.
Using a finite partition of unity we may assume that the support $\supp(f)$ is contained in some rectangle
$\icc{d}{a}\times\icc{b}{c}$ with $(a,b,c,d)$ positive.

\vskip.1cm
\noindent
\begin{minipage}{.65\textwidth}
Fix an interval $\icc{a_0}{b_0}\subset\ioo{a}{b}$ and observe that
for every open set $J\subset\ioo{b_0}{a_0}$ the set
\bqn
S_J:=\{\xi\in J\cap S_\Gamma:\,\overline\xi\in\icc{a_0}{b_0}\}
\eqn
is dense in $J$.  Choose then $a'\in S_{(a,a_0)}$, $b'\in S_{(b_0,b)}$,
$\{c',d'\}\subset S_{(c,d)}$ with $(a',b',c',d')$ positive 
and observe that the quadruple $(a',b',c',d')$ is good.
\end{minipage}
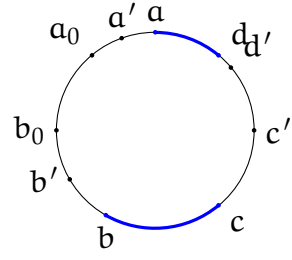
\begin{minipage}{.35\textwidth}
\begin{center}
\begin{tikzpicture}[scale=.65]
\draw (0,0) circle [radius=2cm];
\filldraw[blue] (0,2) circle [radius=1pt] node[above, black] {$a$};
\filldraw (-.68,1.88) circle [radius=1pt] node[above] {$a'$};
\filldraw (-1.28,1.53) circle [radius=1pt] node[above left] {$a_0$};
\filldraw (-2,0) circle [radius=1pt] node[left] {$b_0$};
\filldraw (-1.73,-1) circle [radius=1pt] node[left] {$b'$};
\filldraw[blue]  (-1,-1.73) circle [radius=1pt] node[below, black] {$b$};
\filldraw[blue]  (1.28,-1.53) circle [radius=1pt] node[below right, black] {$c$};
\filldraw (2,0) circle [radius=1pt] node[right] {$c'$};
\filldraw (1.53,1.28) circle [radius=1pt] node[above right] {$d'$};
\filldraw[blue]  (1.28,1.53) circle [radius=1pt] node[above right, black] {$d$};
\draw[blue, very thick] (1.28,1.53) arc[start angle=50, end angle=90, radius=2];
\draw[blue, very thick] (1.28,-1.53) arc[start angle=310, end angle=240, radius=2];
\end{tikzpicture}
\end{center}
\end{minipage}

Fix some distance inducing the topology on $(\deH)^{(2)}$, fix $\epsilon>0$ and let $\delta>0$ be such that 
if $R\subset(\deH)^{(2)}$ is any closed rectangle of diameter
$\operatorname{diam}(R)<\delta$, then 
\bqn
\max_Rf-\min_Rf<\epsilon\,.
\eqn

Fix now a cover of $\supp (f)$ by closed rectangles $R_i=\icc{x_i}{y_i}\times\icc{z_i}{w_i}$,
for $i=1,2,\dots,N$, such that 
\be
\item $\operatorname{diam}(R_i)<\delta$;
\item the interiors of the rectangles are pairwise disjoint;
\item $R_i\subset\ioo{d'}{a'}\times\ioo{b'}{c'};$
\item $\{x_i,y_i\}\subset S_{(d',a')}$ and $\{z_i,w_i\}\subset S_{(b',c')}$.
\ee
Observe that for every $1\leq i\leq N$, the quadruple $(x_i,z_i,w_i,y_i)$ is good.
As a result, we have (see Lemma~\ref{lem:new4.12})
\be
\item[(5)] $\mu(R_i)=[x_i,z_i,w_i,y_i], \quad \mu_n(R_i)=[x_i,z_i,w_i,y_i]_n$, and
\item[(6)] $\mu(R_i\cap R_j)=0, \quad \mu_n(R_i\cap R_j)=0$ for all $i\neq j$. 
\ee
It follows then that 
\bqn
\ba
\left|\int f\,d\mu_n-\sum_{i=1}^N(\min_{R_i}f)\mu_n(R_i)\right|
&\leq\sum_{i=1}^N\epsilon\mu_n(R_i)\\
&\leq\epsilon\mu_n(\icc{d'}{a'}\times\icc{b'}{c'})\\
&=\epsilon[a',b',c',d']_n
\ea
\eqn
and similarly 
\bqn
\ba
\left|\int f\,d\mu-\sum_{i=1}^N(\min_{R_i}f)\mu(R_i)\right|
\leq\epsilon[a',b',c',d']\,.
\ea
\eqn
From these inequalities, the assumption that $\lim_n[\,\cdot\,,\,\cdot\,,\,\cdot\,,\,\cdot\,]_n=[\,\cdot\,,\,\cdot\,,\,\cdot\,,\,\cdot\,]$
and (5), we deduce that 
\bqn
\overline{\lim_{n\to\infty}}\left|\int f\,d\mu-\int f\,d\mu_n\right|
\leq2\epsilon[a',b',c',d']
\eqn
and hence
\bqn
\lim_{n\to\infty}\int f\,d\mu_n=\int f\,d\mu\,.
\eqn
\end{proof}

\subsection{Integral crossratios}\label{s.3.5}
The goal of the section is to prove the following
\begin{prop}\label{p.atomic}
If the positive crossratio 
$$[\,\cdot\,,\,\cdot\,,\,\cdot\,,\,\cdot\,]\colon X^{[4]}\to[0,\infty)$$
takes values in the integers $\bN=\{0,1,2,\ldots\}$, then the associated current  corresponds to an integral  geodesic multicurve.
\end{prop}
\begin{proof}
We first show that $\mu$ is purely atomic. It follows from Proposition \ref{prop:Rm} (3) and (4) that $\mu(R)\in\bN$ for all proper open and proper closed rectangles. Let $g\in\supp (\mu)\subset (\partial\H)^{(2)}$ and let $R_n$ be a decreasing sequence of open  rectangles with $\bigcap R_n=\{g\}$. Since $\mu(R_n)\in\bN$ we have either 
\begin{enumerate}
\item $\mu(R_n)\geq 1$ for all $n\geq 1$ and $\mu(\{g\})\geq 1$
\end{enumerate}
or
\begin{enumerate}
\item[(2)] there exists $n_0$ such that $\mu(R_n)=0$ for all $n\geq n_0$.
\end{enumerate} 
The second case cannot happen since $g$ is in the support of $\mu$. As a result $\mu$ is purely atomic with $\bN$-valued atoms.

We now show that all geodesics in the support of $\mu$ are either closed or connect two cusps. 
Let $(a,b)\in (\partial\H)^{(2)}$ be such an atom. Then $\Gamma\cdot(a,b)$ meets every compact subset of $(\partial\H)^{(2)}$ in only finitely many points. 
As a result, if $g\subset\H$ is the geodesic connecting $(a,b)$, $\pr(g)\subset \Sigma$ is a closed subset  
where, as always, $\pr\colon\H\to\Sigma$ denotes the universal covering map. 
Thus either $g$ corresponds to a periodic geodesic, or $\pr(g)$ is a geodesic connecting two cusps. 

There is a compact subset $K\subset \Sigma$ such that every biinfinite geodesic, 
as well as every closed geodesic, meets $K$.
Thus if $\calA$ denotes the set of atoms of $\mu$ there is a compact subset $C\subset (\partial\H)^{(2)}$ such that for all $a\in\calA$, 
$\Gamma\cdot a\cap C\neq \emptyset$. This implies that
$$\mu=\sum_{c\in F}n_c\delta _c$$
where
$F$ is a finite set of geodesics either periodic or connecting to cusps, 
$\delta_c$ is the geodesic current corresponding to $c$, 
and $n_c\in\bN^*$.
\end{proof}

\subsection{Restriction to a subsurface}\label{s.rest}

We conclude the section discussing how the construction of the geodesic current associated to a positive crossratio behaves with respect to restriction to subsurfaces. This will be useful in the study of maximal representations. 

Let $\Sigma'\subset \Sigma$ be a subsurface with geodesic boundary.
Let 
$\calG(\Sigma')\subset (\partial \H)^{(2)}$ be
the set of geodesics 
whose projection lies in the interior $\mathring\Sigma'$ of $\Sigma'$.
%
If $\mu$ is a current on $\Sigma$, 
we define $\mu_{|\Sigma'}\in \Curr(\Sigma)$ by
\bqn
\mu_{|\Sigma'}:=\chi_{\calG(\Sigma')}\mu\,,
\eqn
where $\chi_{\calG(\Sigma')}$ is the characteristic function of
$\calG(\Sigma')$.  

We write $i(\mu,\partial \Sigma')=0$ when $i(\mu,c)=0$ for every
boundary component $c$ of $\Sigma'$. This is the case precisely when
no geodesic in the support of $\mu$ intersects  $\partial \Sigma'$;
thus in that case we have
\bqn
i(\mu, c)=i(\mu_{|\Sigma'}, c)
\eqn
for every closed geodesic $c$ contained in $\Sigma'$.

We now choose a finite area hyperbolization 
$\Sigma_0=\Gamma_0\backslash\H$ of
$\mathring\Sigma'$ and a corresponding identification
$h\colon\Gamma_0\to \pi_1(\Sigma')<\Gamma$.
We denote by $\phi\colon\deH\to \deH$ an injective, monotone $h$-equivariant map. This is a quasi-conjugacy that opens all the cusps corresponding to geodesic boundary components of $\Sigma'$. It follows from this discussion that:

\begin{prop}
   Let $\mu$ be a current on $\Sigma$.
Then 
\bqn
\mu_0(A):=\mu((\phi\times\phi)(A))
\eqn
defines a current $\mu_0=\phi^*\mu$ on $\Sigma_0$, that we will call
the current {\em induced by} $\mu$ on $\Sigma_0$.
We have the following properties :
\begin{enumerate}
\item $\phi^*\mu=\phi^*\mu_{|\Sigma'}$ .
\item If $\mu_{|\Sigma'}'$ has compact carrier included in $\mathring\Sigma'$, 
then $\mu_0$ has compact carrier in $\Sigma_0$ .
\item Let $\mu$, $\nu$ be currents on $\Sigma$. Assume that 
$\mu_{|\Sigma'}$ and $\nu_{|\Sigma'}$ have compact carrier
included in $\mathring\Sigma'$. Then 
\bqn
i(\mu_{|\Sigma'},\nu_{|\Sigma'})=i(\mu_0,\nu_0).
\eqn
\item  
We have 
\bqn
i(\mu_0, \gamma)=i(\mu_{|\Sigma'}, h(\gamma))
\eqn
for all hyperbolic $\gamma\in\Gamma_0$.
In particular if $i(\mu,\partial \Sigma')=0$ then
$i(\mu_0, \gamma)=i(\mu, h(\gamma))$ 
for all hyperbolic $\gamma\in\Gamma_0$.
 
\item Assume that $i(\mu,\partial \Sigma')=0$. 
If $\mu$ is the current associated 
to a positive crossratio $b\in\calCR(\hg)$, then $\mu_0$
is the current associated 
to the positive crossratio $b_0=\phi^*b$ in $\calCR(H_{\Gamma_0})$.
\end{enumerate}
\end{prop}

\section{Equivariant tree embeddings}\label{s.tree}
In this section we discuss barycenter maps compatible with positive crossratios and prove Theorem \ref{thm_intro:thm7}. In \S\ref{s.acttree} we discuss a first class of actions to which Theorem \ref{thm_intro:thm7} applies: framed actions on trees.
\subsection{Actions with  compatible crossratio and barycenter}
\label{sec:5}
Let $\rho\colon\Gamma\to\Isom(\calX)$ be an action by isometries on a metric space $(\calX,d_\calX)$
and $ X\subset\deH$ be a non-empty $\Gamma$-invariant subset.

\begin{defn}\label{d.bary} \be
\item A {\em $\calX$-valued barycenter map}  on $X$ 
(or just a {\em barycenter map} if the context is clear)
is a map 
\bqn
\beta\colon
 X^{(3)}\longrightarrow \calX
\eqn
defined on  the set $X^{(3)}$ of triples of distinct points in $ X$ that verifies:
\be
\item $\beta$ is $S_3$-invariant;
\item $\beta$ is $\Gamma$-equivariant.
\ee
\item  A barycenter map is {\em compatible with a positive crossratio} $[\,\cdot\,,\,\cdot\,,\,\cdot\,,\,\cdot\,]$
defined on $ X$ if, whenever $(a,b,c,d)\in X^{[4]}$,
\bqn
[a,b,c,d]=d_\calX(\beta(a,b,d),\beta(a,c,d))\,.
\eqn
\ee
\end{defn}
Condition (2) is inspired by the construction of a crossratio induced by a framed action on a tree, as in Example \ref{ex:trees}.  Indeed we have:

\begin{example}
Let $\rho\colon\Gamma\to \Isom(\mathcal T)$ be a framed action of $\Gamma$ on a real tree $\mathcal T$ 
with framing $\varphi\colon X\to \partial_\infty\mathcal T$. 
Then the usual barycenter $B\colon\partial_\infty\mathcal T^{(3)}\to \mathcal T$ 
induces a barycenter map on $X$ compatible with the positive crossratio $[\,\cdot\,,\,\cdot\,,\cdot\,,\,\cdot\,]_\varphi$.
\end{example}

The goal of the section is to prove the following:
\begin{thm}\label{thm:5.2} 
Let $\rho\colon\Gamma\to\Isom(\calX)$ be an isometric action. Assume that there is a positive crossratio
on $ X$ and a compatible barycenter map,
 that the geodesic current $\mu$ associated to the positive crossratio is of lamination type
and let $\calV(\mu)$ be the set of vertices of the $\bR$-tree $\calT(\mu)$ associated to $\mu$.
Then there is an equivariant isometric embedding 
\bqn
\calV(\mu)\hookrightarrow \calX\,.
\eqn
%
%
\end{thm}

In order to define the embedding, we will show that  each complementary region
of the geodesic lamination $\widetilde\calL:=\supp(\mu)$ leads to  a well-defined barycenter.  More precisely, given a complementary region $\bRr$ of $\widetilde\calL$,
we will show in Proposition~\ref{prop: barycenter map factorizes} that the map that to three points $a,b,c\in(\deH\smallsetminus\bRr(\infty))\cap X$
associates their barycenter is constant if the points are in different connected components of $\deH\smallsetminus\bRr(\infty)$.
Here $\bRr(\infty)$ is the intersection of $\deH$ with the closure of $\bRr$ in $\overline{\H}$.
 In Lemma~\ref{lem:5.6} we show that the obtained map is isometric.

As a first step in the proof of Proposition~\ref{prop: barycenter map factorizes}, we show that the crossratio of 4-tuples separated by the lamination vanishes:
\begin{lem}\label{lem:5.3}  Let $\bRr$ be a complementary region of $\widetilde\calL$ and 
let $(a,b,c,d)$ be positively oriented such that $\{a,b,c,d\}\subset\deH\smallsetminus\bRr(\infty)$
and $\{a,b\}$, as well as $\{c,d\}$, are in different connected components of $\deH\smallsetminus\bRr(\infty)$.
Then 
\bqn
\mu(\icc{d}{a}\times\icc{b}{c})=0\,.
\eqn  
In particular, if in addition $(a,b,c,d)\in X^{[4]}$ 
\bqn
[a,b,c,d]=0.
\eqn
\end{lem}

\begin{proof}  Since $\supp(\mu)=\widetilde{\calL}$, we have that $\mu((\deH)^{(2)}\smallsetminus\widetilde\calL)=0$.
Thus it suffices to show that under the hypotheses of the lemma
\bqn
\icc{d}{a}\times\icc{b}{c}\subset(\deH)^{(2)}\smallsetminus\widetilde\calL\,.
\eqn
Assume there is a geodesic $g\in\widetilde\calL$ connecting $\icc{d}{a}$ to $\icc{b}{c}$. 
Then $\bRr$ must be contained in one of the half planes determined by $g$ and hence
either $\{a,b\}$ or $\{c,d\}$ are contained in the same connected component of 
$(\deH)^{(2)}\smallsetminus\bRr(\infty)$, contradicting the hypothesis.
The second statement follows directly from Proposition \ref{prop:Rm} (1).
\end{proof}

With the use of Lemma~\ref{lem:5.3} we proceed to define the barycenter of a complementary region of $\widetilde\calL$.

\begin{prop}\label{prop: barycenter map factorizes}
Let $\bRr$ be a complementary region of $\widetilde\calL$.
Then $\beta(a,b,c)$ is independent of the choice of $\{a,b,c\}\subset X$,
provided $a,b,c$ lie in three distinct connected components of $\deH\smallsetminus\bRr(\infty)$.
\end{prop}

\begin{proof}  We split the proof in three easy steps.

(1) {\em Given three connected components $I_1,I_2,I_3$ of $\deH\smallsetminus\bRr(\infty)$,
$\beta(a,b,c)$ is independent of the choices $a\in I_1\cap  X$, $b\in I_2\cap  X$
and $c\in I_3\cap  X$.}

Indeed, given $\{a,a'\}\subset I_1\cap  X$, we may assume, modulo exchanging $b$ and $c$,
and also $a$ and $a'$,
that $(c,a,a',b)\in X^{[4]}$.  By Lemma~\ref{lem:5.3} we have then that
\bqn
d_X(\beta(c,a,b),\beta(c,a',b))=[c,a,a',b]=0\,.
\eqn
Thus $\beta(a,b,c)=\beta(a',b,c)$.  

Given now $b'\in I_2\cap  X$ and $c'\in I_3\cap  X$,
we conclude, using the $S_3$-invariance of $\beta$, that 
\bqn
\beta(a',b',c')=\beta(a',b',c)=\beta(a',b,c)=\beta(a,b,c).
\eqn

(2) {\em Let $I_1,I_2,I_3,I_4$ be distinct components of $\deH\smallsetminus\bRr(\infty)$,
$\{a,b,c,c'\}\subset X$ with $a\in I_1$, $b\in I_2$, $c\in I_3$ and $c'\in I_4$.
Then $\beta(a,b,c)=\beta(a,b,c')$.}

We distinguish two cases.

(2.a) If $c,c'$ are in the same connected component of $\deH\smallsetminus\{a,b\}$,
then possibly upon permuting $a,b$ and $c,c'$, we may assume that $(b,c,c',a)\in X^{[4]}$.
By Lemma~\ref{lem:5.3} this implies that
\bqn
d_X(\beta(b,c,a),\beta(b,c',a))=[b,c,c',a]=0\,.
\eqn
Thus $\beta(b,c,a)=\beta(b,c',a)$.

(2.b) If instead $c,c'$ are in distinct connected components of $\deH\smallsetminus\{a,b\}$,
we may assume, possibly permuting $a,b$ and $c,c'$, that $(a,c',b,c)$ is positively oriented.
Using (2.a) in the second and fourth equality we obtain
\bqn
\beta(a,b,c')=\beta(a,c',b)=\beta(a,c',c)=\beta(c,a,c')=\beta(c,a,b)=\beta(a,b,c)\,,
\eqn
which shows the assertion.

(3) We finish now the proof of the proposition.  Let $\{a,b,c,a',b',c'\}\subset X$
with $a\in I_1$, $b\in I_2$, $c\in I_3$, $a'\in I_1'$, $b'\in I_2'$ and $c\in I_3'$,
where $I_1, I_2, I_3$ and $I_1', I_2', I_3'$ are distinct connected components of $\deH\smallsetminus\bRr(\infty)$. We can assume, up to reordering the indices that $I_j\neq I_k$ for $j\neq k$.
Then it follows from (1) and (2) that 
\bqn
\beta(a,b,c)=\beta(a',b,c)=\beta(a',b',c)=\beta(a',b',c')\,.
\eqn
\end{proof}

\begin{defn}  The \emph{barycenter} $\beta(\bRr)$ of a complementary region $\bRr$ of $\widetilde\calL$ is the point $\beta(a,b,c)$ for any choice $\{a,b,c\}\subset X$
of points lying in distinct components of $\deH\smallsetminus\bRr(\infty)$.
\end{defn}

Taking into account the discussion in  \S\ref{s.treemu}, the following lemma concludes the proof of Theorem~\ref{thm:5.2}. 

\begin{lem}\label{lem:5.6}  For the distance $d_{\mu}$ on the set
 $\calV(\mu)$ 
 of complementary regions of $\widetilde\calL$,
 we have
\bqn
d_\calX(\beta(\bRr_1),\beta(\bRr_2))=d_{\mu}(\bRr_1,\bRr_2)
\eqn
for all $\bRr_1,\bRr_2\in\calV(\mu)$. 
\end{lem}

\begin{proof}  Let $(x_1,y_1)$ be the endpoints of the geodesic in $\partial\bRr_1$ separating $\bRr_1$ from $\bRr_2$
and $(x_2,y_2)$ the endpoints of the geodesic in $\partial\bRr_2$ separating $\bRr_2$ from $\bRr_1$,
ordered so that $(x_1,y_1,x_2,y_2)\in X^{[4]}$.
Choose $a,b\in\ioo{x_1}{y_1}$ in different connected components of $\deH\smallsetminus\bRr_1(\infty)$
and $c,d\in\ioo{x_2}{y_2}$ in different connected components of $\deH\smallsetminus\bRr_2(\infty)$
in such a way that $(a,b,c,d)\in X^{[4]}$.

\medskip
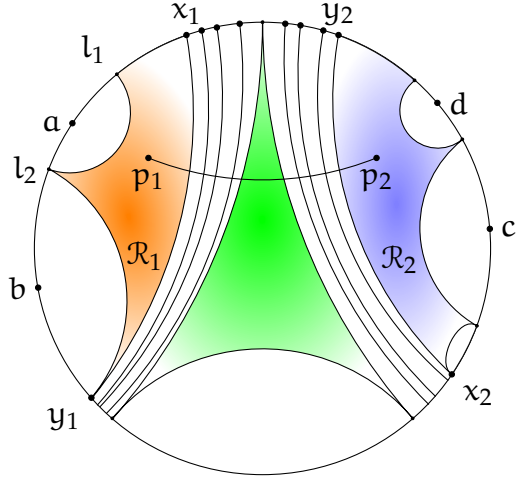
\begin{center}
\begin{tikzpicture}
\draw (0,0) circle [radius=3];
\shadedraw[inner color=orange, outer color=white, draw=black, shading=radial, shading angle=100]
	(-1,2.83) arc[start angle=370, end angle=321, radius=6] --
	(-2.25, -1.98) arc [start angle=320, end angle=421, radius=2] --
	(-2.81,1.05) arc [start angle=250, end angle=399, radius=.8] --
	(-1.915,2.31) arc [start angle=129.5, end angle=109.2, radius=3] -- cycle;
\draw (-1.55,-.1) node {$\mathcal R_1$};
\filldraw (-2.25, -1.98) circle [radius=1pt] node[below left] {$y_1$};
\filldraw (-1,2.83)  circle [radius=1pt] node[above] {$x_1$};
\filldraw (-1.915,2.31) node [above left]{$l_1$} circle [radius=.5pt];
\filldraw (-2.81,1.05) node [left]{$l_2$} circle [radius=.5pt];

\shadedraw[inner color=green, outer color=white, draw=black, shading=radial]
	(0,3) arc [start angle=360, end angle=318.9, radius=8] --
	(-1.975,-2.253) arc [start angle=140, end angle =40, radius=2.58] --
	(1.975,-2.253) arc [start angle=221.1, end angle=180, radius=8] -- cycle;
\filldraw (-1.975,-2.253) circle [radius=.5pt];
\filldraw (1.975,-2.253) circle [radius=.5pt];
\filldraw (0,3) circle [radius=.5pt];


\shadedraw[inner color=blue!50, outer color=white, draw=black, shading=radial]
	(1,2.83) arc [start angle=170, end angle=226.7, radius=5] --
	(2.49,-1.67) arc [start angle=326, end angle=340, radius=3] --
	(2.82,-1.026) arc [start angle=240, end angle=129, radius=1.5] --
	(2.628,1.447) arc [start angle=310, end angle=128, radius=.5] --
	(2,2.229) arc [start angle=48, end angle=70, radius=3] -- cycle;
	
\filldraw (1,2.83) circle [radius=1pt] node[above] {$y_2$};
\filldraw (2.49,-1.67)  circle [radius=1pt] node[below right]{$x_2$};
\filldraw (2.82,-1.026)  circle [radius=.5pt];
\filldraw (2.628,1.447)  circle [radius=.5pt];
\filldraw (2,2.229) circle [radius=.5pt];

\draw (1.8,-.15) node {$\mathcal R_2$};
\filldraw (-2.5,1.66) circle [radius=1pt] node[left] {$a$};
\filldraw (-2.95,-0.52) circle [radius=1pt] node[left] {$b$};
\filldraw (2.99,0.26) circle [radius=1pt] node[right] {$c$};
\filldraw (2.3,1.93) circle [radius=1pt] node[right] {$d$};

\draw (2.49,-1.67) to [out=145,in=145] (2.82,-1.026);

\filldraw (-1.5,1.2) circle [radius=1pt] node[below] {$p_1$};
\filldraw (1.5,1.2) circle [radius=1pt] node[below] {$p_2$};
\draw (-1.5,1.2) arc[start angle=248, end angle=292, radius=4];

\filldraw (-.8,2.89) circle [radius=1pt];
\draw (-.8,2.89) arc[start angle=370, end angle=319.1, radius=6];
\filldraw (-.6,2.93) circle [radius=1pt];
\draw (-.6,2.93) arc[start angle=369, end angle=317, radius=6];
\filldraw (-.3,2.98) circle [radius=1pt];
\draw (-.3,2.98) arc[start angle=366.5, end angle=316, radius=6.4];
\filldraw (.8,2.89) circle [radius=1pt];
\draw (.8,2.89) arc [start angle=170, end angle=229, radius=5];
\filldraw (.5,2.958) circle [radius=1pt];
\draw (.5,2.958) arc [start angle=174, end angle=225.7, radius=6];
\filldraw (.3,2.98) circle [radius=1pt];
\draw (.3,2.98) arc [start angle=176, end angle=225, radius=6.5];
\end{tikzpicture}
\end{center}

\noindent
Then
\bqn
d_\calX(\beta(\bRr_1),\beta(\bRr_2))
=d_\calX(\beta(a,b,d),\beta(a,c,d))
=[a,b,c,d]\,.
\eqn
Since the geodesics bounding $\bRr_1$ and $\bRr_2$ are all  $\mu$-short, 
we have 
\bqn
\mu(\{a\}\times\icc{b}{c})\leq \mu(\ioo{l_1}{l_2}\times\ioo{l_2}{l_1})=0\,,
\eqn
where $\{l_1,l_2\}$ is the geodesic in $\partial\bRr_1$ separating $a$
from $\bRr_1$, and similarly $\mu(\{d\}\times\icc{b}{c})=0$.
As a result, from
Proposition \ref{prop:Rm} (\ref{it: mu(int) leq cr leq mu(closure)})
we have the equality
\bqn
\mu(\icc{d}{a}\times\icc{b}{c})=[a,b,c,d]=\mu(\ioo{d}{a}\times\ioo{b}{c})\,.
\eqn
If now  $p_i\in\bRr_i$, the set of leaves in $\widetilde\calL$ that intersect
the segment $(p_1,p_2)$ is exactly the set of leaves in $\widetilde\calL$
that separate $\bRr_1$ from $\bRr_2$.
This is also the same as the set of leaves in $\widetilde\calL$ that connect $\icc{d}{a}$ to $\icc{b}{c}$.
The assertion then follows from the above considerations,
recalling that $d_{\mu}(\bRr_1,\bRr_2)$ is the measure of this set,
\bqn
d_\calX(\beta(\bRr_1),\beta(\bRr_2))=[a,b,c,d]=\mu(\icc{d}{a}\times\icc{b}{c})=d_{\mu}(\bRr_1,\bRr_2)\,.
\eqn
\end{proof}

\subsection{Framed actions on trees}\label{s.acttree}
Theorem \ref{thm:5.2} applies to framed actions on trees:
\begin{prop}
\label{thm: framed trees}  
Let $\rho\colon\Gamma\to\Isom(\calT)$ be an action by isometries
on an real tree $\calT$
with a framing $\varphi\colon X\to \partial_\infty\calT$. 
Suppose that the associated crossratio 
$[\,\cdot\,,\,\cdot\,,\,\cdot\,,\,\cdot\,]_\varphi$
is positive, and denote by
$\mu_\rho$ the associated geodesic current.
Then $\mu_\rho$ 
corresponds to a measured lamination, 
and there is a $\Gamma$-equivariant isometric embedding
\bqn
\calT(\mu_\rho)\hookrightarrow \calT
\eqn
In particular, for all hyperbolic $\gamma\in\Gamma$,
$$\ell_\calT(\rho(\gamma))=i(\mu_\rho,\gamma).$$
\end{prop}

\begin{proof}
By 
Proposition \ref{prop: cr from tree is ultram},
the crossratio $[\,\cdot\,,\,\cdot\,,\,\cdot\,,\,\cdot\,]_\varphi$
is ultrametric, hence by 
Proposition \ref{prop: ultram cr give lamin}, the current
$\mu_\rho$ is of lamination type.

Now define the barycenter $\beta(x,y,z)$ 
of $(x,y,z)\in X^{(3)}$ as the
barycenter $\beta_\calT(\varphi(x),\varphi(y),\varphi(z))$
in the tree $\calT$ of 
$(\varphi(x),\varphi(y),\varphi(z))$. 
Then $\beta$ is by construction a equivariant  barycenter map compatible with
the crossratio (by \eqref{eq: comp barycenter in a tree} in Example \ref{ex:trees} and
$S_3$-invariance of $\beta_\calT$), 
hence induces an equivariant isometry
\bqn
\Psi\colon\calV(\mu_\rho)\hookrightarrow \calT
\eqn
of the vertices of the dual tree $\calT(\mu_\rho)$  
by Theorem \ref{thm:5.2}. 
Since $\calT$ is uniquely geodesic, we can extend  $\Psi$ to $\calT(\mu_\rho)$.
Then for all  $\gamma\in\Gamma$ we have 
$\ell_\calT(\rho(\gamma))=\ell_{\Psi(\calT(\mu_\rho))}(\gamma)$ 
as $\Psi(\calT(\mu_\rho))$ is a convex subset of $\calT$,
and 
$\ell_{\Psi(\calT(\mu_\rho))}(\gamma)=\ell_{\calT(\mu_\rho)}(\gamma)$ 
as $\Psi$ is isometric.
Now for hyperbolic $\gamma$ representing a closed geodesic $c$
we have $\ell_{\calT(\mu_\rho)}(\gamma)=i(\mu_\rho,\delta_c)$,
%
%
hence
$\ell_\calT(\rho(\gamma))=i(\mu_\rho,\gamma)$.
\end{proof}

\section{The geometry of the Siegel spaces over real closed fields}\label{sec:max_repr_van_sys}

The goal of this section is to recall facts about the geometry of Siegel spaces over real closed fields 
needed to show that maximal framed actions give rise to a positive crossratio and admit a compatible barycenter 
(Lemma \ref{lem:preconseq-of-causal} and Proposition \ref{prop: barycenter is symmetric}). 

\subsection{Real closed fields}\label{subsec:rcf}
Recall that an ordered field is a field $\bF$ endowed with a total order relation $\leq$ satisfying:
\be
\item if $x\leq y$ then  $x+z\leq y+z$ for all $z\in\bF$;
\item if $0\leq x$ and $0\leq y$, then $0\leq xy$.
\ee
The fields $\bQ$ and $\bR$ with their usual order are examples; 
while some fields admit no ordering, like $\bC$, others admit many, like $\bR(X)$.

\begin{example}  The orders on $\bR(X)$ admit the following description for $\epsilon\in\{-,+\}$:
\be
\item  $>_{\epsilon\infty}$:  if $f\in\bR(X)$, $f>_{\epsilon\infty}0$
if $f(t)>0$ for $t\to\epsilon\infty$.
\item $>_{a\epsilon}$ for $a\in\bR$:  we say that $f>_{a+}0$ if there exists $\eta=\eta(f)\in\bR$, $\eta>0$, with  $f(t)>0$ on the interval $(a,a+\eta)$,
and $f>_{a-}0$ if $f(t)>0$ on the interval $(a-\eta,a)$.
\ee
\end{example}

A basic fact is that every ordered field $\bF$ admits a real closure $\overline{\bF}^\mathrm{r}$,
that is a maximal algebraic extension of $\bF$ to which the order extends.  
Such a real closure is then unique up to a unique $\bF$-isomorphism.
An ordered field $\bF$ is then {\em real closed} if the ordering does not extend to any proper algebraic extension.
Two useful characterizations are the following:
\be
\item the field $\bF$ is ordered and $\bF(\imath)$ is algebraically closed, with $\imath=\sqrt{-1}$;
\item the field $\bF$ is ordered, every positive element is a square and any odd degree polynomial has a root. 
\ee

Real closed fields have the same first order logic as the field $\bR$ of the reals.
An important consequence that is implicit in most of the geometric properties of the Siegel space we use,
is that any symmetric matrix with coefficient in a real closed field is orthogonally similar to a diagonal one.

An \emph{$\bR$-valued valuation} is a map 
$$v\colon\bF\to\bR\cup\{\infty\},$$
where $v\colon(\bF^\times,\cdot)\to(\bR,+)$ is a group homomorphism,  $v(0)=\infty$, and,
if we define the \emph{norm} of $x\in\bF$ by $\|x\|_v:=e^{-v(x)}$ if
$x\neq 0$ and  $\|0\|_v:=0$, then $$\|x+y\|_v\leq\|x\|_v+\|y\|_v.$$
An important valuation to keep in mind is $-\ln(|\cdot|)\colon\bR\to \bR\cup\{\infty\}$. 
Indeed valuations offer a replacement for the logarithm in more general real closed fields. 
The valuation is {\em order compatible} if $|x|\leq|y|$ implies that $v(x)\geq v(y)$,
and it is {\em non-Archimedean} if $v(x+y)\geq\min\{v(x),v(y)\}$, 
that is if $\|x+y\|_v\leq\max\{\|x\|_v,\|y\|_v\}$.

\begin{example}\label{e.real closed} The following are examples of ordered fields:
\be
\item The field $\bR$ of real numbers and the field $\overline{\bQ}^r$ of real algebraic numbers;
both are real closed.
\item Let $\omega$ be a non-principal ultrafilter on $\bN$.  The
  quotient $\bR_\omega$ of the ring $\bR^\bN$ by the 
equivalence relation $(x_n)\sim(y_n)$ if $\omega(\{n:\,x_n=y_n\})=1$,
ordered in such a way that  positive elements are the classes of the sequences
$(x_n)$ such that $\omega(\{n:\,x_n> 0\})=1$, is a real closed field called  the field of the {\em
  hyperreals}.  It does not admit any order compatible $\bR$-valued valuation.

\item Let $\sigma\in\bR_\omega$ be a positive infinitesimal, that is
$\sigma$ can be represented by a sequence $(\sigma_n)_{n\geq0}$
with $\lim\sigma_n=0$ and $\sigma_n>0$.  Then 
\bqn
\Oo_\sigma:=\left\{x\in\bR_\omega:\,|x|<\sigma^{-k}\text{ for some }k\in\bZ\right\}
\eqn
is a valuation ring with maximal ideal 
\bqn
\Ii_\sigma=\left\{x\in\bR_\omega:\,|x|<\sigma^{k}\text{ for all }k\in\bZ\right\}\,.
\eqn
The quotient $\bR_{\omega,\sigma}:=\Oo_\sigma/\Ii_\sigma$ is a real closed field, called the {\em Robinson field}.
It admits an order compatible valuation
\bqn
v(x)=\lim_\omega\frac{\ln|x_n|}{\ln\sigma_n}\,,
\eqn
where $(x_n)_{n\geq0}$ represents $x$,
that leads to a non-Archimedean norm
\bqn
\|x\|_v:=e^{-v(x)}\,.
\eqn
\item Let $G$ be a totally ordered Abelian group and let
\bqn
\mathcal{H}(G):=
\left\{f=\sum_{\sigma\in G} a_\sigma X^\sigma:\,a_\sigma\in\bR \text{ and }\supp(f)\subset G\text{ is well ordered}\right\}
\eqn
be the set of formal power series with exponents in $G$ and coefficients in $\bR$, 
where $\supp(f):=\{\sigma\in G:\,a_\sigma\neq0\}$.
This is an $\bR$-vector space and the restriction on supports allows one to define
a ring structure that extends the ordinary multiplication on the group ring $\bR[G]$ of $G$.
In fact $\calH(G)$ turns out to be a field, called the {\em Hahn field with exponents $G$}.
It is ordered by setting  $f>0$ if $a_{\sigma_0}>0$, where $\sigma_0=\min(\supp(f))$
and it admits a $G$-valued compatible valuation $v(f)=\sigma_0$.
Moreover $\mathcal{H}(G)$ is real closed if $G$ is divisible.  
For all of the above statements, see \cite[Theorem~2.15]{Dales_Woodin}.

\item  (Compare with \S~\ref{ex:strubel})  If $\alpha\in\bR\smallsetminus\bQ$, then the ring morphism 
\bqn
\ba
\bR[x,y]&\longrightarrow\,\,\,\mathcal{H}(\bR)\\
P\quad&\longmapsto P(x,x^\alpha)
\ea
\eqn
extends to $\bR(x,y)$.  In this way we obtain an order on $\bR(x,y)$ for which
$P\in\bR(x,y)$ is positive if and only if for some $\epsilon>0$, $P(t,t^\alpha)>0$ for all $t\in(0,\epsilon).$
\ee
\end{example}

\subsection{The Siegel upper half space and the space $\calB_n^\bF$}\label{subsec:Siegel} 
Let $\bF$ be a real closed field, and  $\imath$ be a square root of $-1$.
Endow $V=\bF^{2n}$ with the standard symplectic form 
\bqn
\left\langle\begin{pmatrix}x_1\\y_1\end{pmatrix},\begin{pmatrix}x_2\\y_2\end{pmatrix}\right\rangle:={}^tx_1y_2- {}^ty_1x_2\,,
\eqn
where $x_i,y_i\in\bF^n$. 
The vector space $\Sym_n(\bF)$ of symmetric matrices  
admits a partial order defined by setting 
\bqn
X\ll Y\quad \text{ if }Y-X\text{ is positive definite.}
\eqn
The {\em Siegel upper half space} is the semialgebraic set
\bqn
\snf:=\{Z=X+\imath Y:\,X,Y\in\Sym_n(\bF)\text{ and }Y\gg0\}
\eqn
on which 
\bqn
\Sp(2n,\bF):=\left\{g=\begin{pmatrix}A&B\\C&D\end{pmatrix}:\,\begin{array}{l}{^t\!A}D-{}^tCB=\Id,\,\\{}^t\!AC={}^tCA,\,\\{}^tBD={}^tDB\end{array}\right\}
\eqn
acts by fractional linear transformations
\bq\label{e.flt}
g_\ast Z:=(AZ+B)(CZ+D)^{-1}\,,
\eq
transitively.  Of course this action descends to an action of $\PSp(2n,\bF)$.
The stabilizer of $\imath \Id_n\in\snf$ in $\Sp(2n,\bF)$ is
\bqn
K=\Sp(2n,\bF)\cap \operatorname{O}(2n)
=\left\{\begin{pmatrix}\hphantom{-}A&B\\-B&A\end{pmatrix}:\,{}^tAA+{}^tBB=\Id_n,\,{}^tAB\text{ is symmetric}\right\}\,.
\eqn
%

If $\bF=\bR$, then $\snr$ is the symmetric space associated to $\PSp(2n,\bR)$.  
If, instead, the real closed field $\bF$ is endowed with an order compatible non-Archimedean valuation $v$, 
then $\PSp(2n,\bF)$ acts by isometries on a  $v(\bF)$-metric space $\calB_n^\bF$: a metric quotient of $\snf$ whose construction we now recall. See \cite{BIPP-ann,BIPP-RSPEC} for generalizations of this construction.

On $\snf$ we define an multiplicative $\bF$-valued distance function  as follows.
Since $\bF$ is real closed, any pair $(Z_1,Z_2)$ with $Z_i\in\snf$, for $i=1,2$, is $\PSp(2n,\bF)$-congruent to a unique pair
$(\imath \,\Id_n,\imath D)$, where $D=\diag(d_1,\dots,d_n)$,
$d_1\geq\cdots\geq d_n\geq 1$ in $\bF$.
We then set 
\bqn
D(Z_1,Z_2):=\prod_{i=1}^n d_i.
\eqn

\begin{prop}
  $D$ is  a $\PSp(2n,\bF)$-invariant
  multiplicative distance function on $\snf$, namely, for all $Z_1,Z_2,Z_3\in\snf$,
  \begin{description}
  \item[(MD1)] $D(Z_1,Z_2)\in \bF_{\geq 1}$, with equality of and
    only if $Z_1=Z_2$ ;
  \item[(MD2)] $D(Z_1,Z_2)=D(Z_2,Z_1)$ ;
  \item[(MD3)] $D(Z_1,Z_2)\leq D(Z_1,Z_3)D(Z_3,Z_2)$ .
  \end{description}
\end{prop}

\begin{proof} (MD1) and (MD2) are clear.
We consider the standard action of $\Sp(2n,\bF)$ on $W=\wedge^n(\bF^{2n})$,
endowed with the standard scalar product,
which is $K$-invariant as $K\subset \operatorname{O}(2n)$.
For
$a=\diag(a_1,\dots,a_n,a_1^{-1},\dots,a_n^{-1})$ with
$a_1\geq\cdots\geq a_n\geq 1$ in $\bF$, we easily see that
\[\max_{w\in W, w\neq 0}\frac{||aw||}{||w||}= \prod_{i=1}^n a_i\]
(the biggest eigenvalue of $a$ in $W$).
For $g \in \Sp(2n,\bF)$
the operator norm of $g$ on $W$ is given by
\bqn
|||g||| := \max_{w\in W, w\neq 0}\frac{||gw||}{||w||}\,.
\eqn
Since $g=kak'$ for some $k,k'$ in $K$ and $a$ as before (by the Cartan
decomposition), and 
\bqn
D(\imath \Id_n, g_\ast \imath \Id_n)=D(\imath \Id_n, a_\ast \imath
\Id_n)= 2|||a|||= 2|||g|||\,,
\eqn
(MD3) follows from submultiplicativity of the operator norm and transitivity of the action of $\Sp(2n,\bF)$.
\end{proof}

On $\snf$ we define an \emph{associated pseudo-distance} $d^1$ as follows.
\begin{equation}\label{e.pdist}
d^1(Z_1,Z_2):=-v(D(Z_1,Z_2)=-\sum_{i=1}^n v(d_i).
\end{equation}
The triangle inequality for $d^1$ comes from (MD3).
We denote by ${\calB_n^\bF}$ the metric quotient of $\snf$ with respect to the pseudo-distance $d^1$.
We will not need it, but note in passing that ${\calB_n^\bF}$ can be identified with  the quotient
$\PSp(2n,\bF)/\PSp(2n,\bU)$, where $\bU:=\{x\in\bF:\,\|x\|_v\leq1\}$.

\subsection{Embedding in $\bK$-Lagrangians}\label{subsubsec:projective-model}
In the classical case, the Borel embedding of $\snr$ into the complex Grassmannian
provides a way to endow $\snr$ with structures defined on the Grassmannian,
such as, for example, a crossratio.  We recall from \cite{BP} the analogous picture in the case
of a general real closed field.

Let $\bK=\bF(\imath)$ be the algebraic closure of $\bF$ and 
let us also denote by $\langle\,\,,\,\,\rangle$ the $\bK$-linear extension of the
standard symplectic form to $\bK^{2n}$
and by 
$\sigma\colon\bK^{2n}\to\bK^{2n}$ the complex conjugation. 
Given matrices $Z_1,Z_2\in M_n(\bK)$, 
we will denote by $\left\langle\begin{pmatrix}Z_1\\Z_2\end{pmatrix}\right\rangle$
the subspace of $\bK^{2n}$ generated by the column vectors.  
We denote by $\calL(\bK^{2n})$ the submanifold of ${\rm Gr}_n(\bK^{2n})$ consisting of subspaces that are isotropic for the form $\langle\,\,,\,\,\rangle$.
The map
\bqn
\ba
\Sym_n(\bK)&\longrightarrow\,\,\,\calL(\bK^{2n})\\
Z\qquad&\longmapsto\left\langle\begin{pmatrix}Z\\\Id_n\end{pmatrix}\right\rangle
\ea
\eqn
gives a 
bijection between $\Sym_n(\bK)$
and the subset of $\calL(\bK^{2n})$ of all Lagrangians transverse to $\ell_\infty:=\left\langle\begin{pmatrix}\Id_n\\0\end{pmatrix}\right\rangle$.
This map intertwines the action of $\PSp(2n,\bK)$ on $\Sym_n(\bK)$ by fractional linear transformations \eqref{e.flt} and the standard action on $\calL(\bK^{2n})$. 
This bijection maps $\snf$ to the projective model
\bqn
\Dd_\bF:=\{L\in\calL(\bK^{2n}):\,-\imath\langle\,\cdot\,,\sigma(\,\cdot\,)\rangle|_{L\times L}\gg0\}
\eqn
and sends $\Sym_n(\bF)$ to 
\bqn
\{\ell\otimes \bK:\,\ell\in\calL(\bF^{2n}),\,\ell\text{ is  transverse to }\ell_\infty\}\,.
\eqn

\subsection{Maximal triples and intervals}\label{subsubsec:max-tr-quadr}
We associate to a triple $(\ell_1,\ell_2,\ell_3)$ of pairwise
transverse Lagrangians in $\calL(\bF^{2n})$
the quadratic form $Q_{(\ell_1,\ell_2,\ell_3)}$ on $\ell_1$
defined by 
\bqn
Q_{(\ell_1,\ell_2,\ell_3)}(v):=\langle v,v'\rangle\,,
\eqn
where $v'\in\ell_3$ is the unique vector such that $v+v'\in\ell_2$.  
If $\ell=\left\langle\begin{pmatrix}X\\\Id_n\end{pmatrix}\right\rangle$ and 
$\ell'=\left\langle\begin{pmatrix}X'\\\Id_n\end{pmatrix}\right\rangle$
are pairwise transverse,
then in the coordinates
\begin{equation}\label{e.frame}
\ba
\bF^n&\longrightarrow \ell\\
w&\mapsto\begin{pmatrix}X\\\Id_n\end{pmatrix}w
\ea
\end{equation}
the quadratic form $Q_{(\ell,\ell',\ell_\infty)}$ is represented by $X'-X$.

Two triples of pairwise transverse Lagrangians $(\ell_1,\ell_2,\ell_3)$ and $(m_1,m_2,m_3)$ are 
$\PSp(2n,\bF)$-congruent if and only if the quadratic spaces
$(\ell_1,Q_{(\ell_1,\ell_2,\ell_3)})$ and $(m_1,Q_{(m_1,m_2,m_3)})$
are isomorphic or equivalently if and only if
$Q_{(\ell_1,\ell_2,\ell_3)}$ and $Q_{(m_1,m_2,m_3)}$ have the same signature \cite[Proposition 2.5]{BP}.
The value of the {\em Maslov cocycle} on $(\ell_1,\ell_2,\ell_3)$ is the signature of $Q_{(\ell_1,\ell_2,\ell_3)}$
\bqn
\tau(\ell_1,\ell_2,\ell_3):=\sign Q_{(\ell_1,\ell_2,\ell_3)}\,.
\eqn
The triple $(\ell_1,\ell_2,\ell_3)$ is {\em maximal} if $\tau(\ell_1,\ell_2,\ell_3)=n$,
the maximal value the Maslov cocycle can take. Similarly we say that a triple $(\ell_1,\ell_2,\ell_3)$ is {\em minimal} if $\tau(\ell_1,\ell_2,\ell_3)=-n$; it is easy to verify that $(\ell_1,\ell_2,\ell_3)$ is maximal if and only if $(\ell_2,\ell_1,\ell_3)$ is minimal.
The group $\PSp(2n,\bF)$ acts transitively on pairs of transverse Lagrangians in $\calL(\bF^{2n})$
and on maximal triples.  

Given $\ell,\ell'\in\calL(\bF^{2n})$, we define the \emph{interval}
\bqn
\ioo{\ell}{\ell'}:=\{m\in\calL(\bF^{2n}):\,(\ell, m,\ell')\text{ is maximal}\}\,.
\eqn

\begin{lem}[{\cite[Lemma 2.10]{BP}}]\label{lem:order}  
Let $\ell=\left\langle\begin{pmatrix}X\\\Id_n\end{pmatrix}\right\rangle$ and 
$\ell'=\left\langle\begin{pmatrix}X'\\\Id_n\end{pmatrix}\right\rangle$.
If the triple $(\ell,\ell',\ell_\infty)$ is maximal,
then 
\bqn
\ioo{\ell}{\ell'}=\left\{\left\langle\begin{pmatrix}Y\\\Id_n\end{pmatrix}\right\rangle:\,X\ll Y\ll X'\right\}\,.
\eqn
\end{lem}

For $X,X'\in\Sym_n(\bF)$ with $X\ll X'$, 
we will also denote by $\ioo{X}{X'}$ the set
\bqn
\ioo{X}{X'}:=\{Y\in\Sym_n(\bF):\,X\ll Y\ll X'\}
\eqn
and set
\bqn
\ioo{X}{\infty}:=\{Y\in\Sym_n(\bF):\,X\ll Y\}\,.
\eqn

\subsection{Crossratios}\label{subsubsec:crossratio}
In this subsection we recall the endomorphism valued
crossratio from \cite[\S~4.1]{BP} on quadruples of Lagrangians. 
This, together with a maximal framing, will allow us in \S~\ref{subsec:6.1} to associate to any maximal framed representation $\rho$ a positive crossratio as  in \S~\ref{sec:poscr}.

Given a quadruple of Lagrangians $(\ell_1,\ell_2,\ell_3,\ell_4)$ in $\calL(\bF^{2n})$ with $\ell_1\cap\ell_2=\ell_3\cap\ell_4=\{0\}$,
their \emph{crossratio} is the endomorphism of $\ell_1$ given by 
\bqn
R(\ell_1,\ell_2,\ell_3,\ell_4)=\p_{\ell_1}^{\|\ell_2}\circ \p_{\ell_4}^{\|\ell_3}|_{\ell_1}\,,
\eqn
where $\p_{\ell_j}^{\|\ell_i}$ denotes the projection of $\bF^{2n}$ to $\ell_j$ parallel to the complementary space $\ell_i$.
One verifies that for all $g\in\Sp(2n,\bF)$
\bqn
R(g\ell_1,g\ell_2,g\ell_3,g\ell_4)=gR(\ell_1,\ell_2,\ell_3,\ell_4)g^{-1}
\eqn
and hence
\bqn
\det R(\ell_1,\ell_2,\ell_3,\ell_4)
\eqn
is $\PSp(2n,\bF)$-invariant.
If, for $j\in\{1,2,3,4\}$, $\ell_j\cap\ell_\infty=\{0\}$, then $\ell_j=\left\langle\begin{pmatrix}X_j\\\Id_n\end{pmatrix}\right\rangle$
for some  $X_j\in\Sym_n(\bF)$. By \cite[Lemma~4.2]{BP}, the matrix representing $R(\ell_1,\ell_2,\ell_3,\ell_4)$ in the basis of $\ell_1$ given by \eqref{e.frame} 
 is 
\begin{equation}\label{eqn:Rinfty}
(X_1-X_2)^{-1}(X_2-X_4)(X_3-X_4)^{-1}(X_1-X_3)\,.
\end{equation}
We will denote such matrix $R(\ell_1,\ell_2,\ell_3,\ell_4)$, with an abuse of notation.

\begin{prop}[{\cite[Lemma 4.4]{BP}}]\label{lem:crBP}  Let $\ell_1,\ell_2,\ell_3,\ell_4,\ell_5$ be pairwise transverse Lagrangians.
Then:
\be
\item 
$
R(\ell_1,\ell_2,\ell_4,\ell_5)=R(\ell_1,\ell_2,\ell_3,\ell_5)R(\ell_1,\ell_3,\ell_4,\ell_5)\,.
$
\item 
$R(\ell_1,\ell_2,\ell_3,\ell_4)\text{ is conjugate to } R(\ell_3,\ell_4,\ell_1,\ell_2)\,.
$
\item 
$
\det R(\ell_1,\ell_2,\ell_4,\ell_5)=\det R(\ell_1,\ell_2,\ell_3,\ell_5)\det R(\ell_1,\ell_3,\ell_4,\ell_5)
$  and\\
$
\det R(\ell_1,\ell_2,\ell_3,\ell_4)=\det R(\ell_3,\ell_4,\ell_1,\ell_2)\,.
$
\item If $(\ell_1,\ell_2,\ell_3,\ell_4)$ is a maximal quadruple, 
then all eigenvalues of $R(\ell_1,\ell_2,\ell_3,\ell_4)$ belong to $\bF$ and are strictly larger than one.  
In particular $\det R(\ell_1,\ell_2,\ell_3,\ell_4)>1$.
\ee
\end{prop}

\subsection{$\bF$-tubes and orthogonal projections}\label{subsubsec:F-tubes}
If $\ell,\ell'\in\calL(\bF^{2n})$, the \emph{$\bF$-tube} determined by $\ell,\ell'$ is the 
semi-algebraic subset of $\snf$ given by the equation
\bqn
\Yy_{\ell,\ell'}:=
\left\{Z\in\snf:\,R\left(\ell,\left\langle\begin{pmatrix}Z\\\Id_n\end{pmatrix}\right\rangle,
								     \left\langle\begin{pmatrix}\overline{Z}\\\Id_n\end{pmatrix}\right\rangle,  \ell'\right)=-\Id_n\right\}\,,
\eqn
where $\overline Z$ is the complex conjugate of $Z$, \cite[\S~4.2]{BP}.

If $\bF=\bR$ and $n=1$, the $\bR$-tube $\Yy_{\ell,\ell'}$ is the geodesic between $\ell$ and $\ell'$
while, for $n\geq1$, $\Yy_{\ell,\ell'}$ is a symmetric subspace of $\snr$ that is a Lagrangian submanifold and is isometric
to the symmetric space associated to $\GL(n,\bR)$.  In general,  for all $g\in\PSp(2n,\bF)$,
\bq\label{eq:inv-tube}
\Yy_{g\ell,g\ell'}=g(\Yy_{\ell,\ell'})
\eq
and if we denote
\bqn
\ell_0:=\left\langle\begin{pmatrix}0\\\Id_n\end{pmatrix}\right\rangle\,,
\eqn
then
\bqn
\Yy_{\ell_0,\ell_\infty}:=\{\imath Y:\,Y\in\Sym_n(\bF),Y\gg0\}\,.
\eqn
We will often write $\Yy_{0,\infty}$ for $\Yy_{\ell_0,\ell_\infty}$.

If $(\ell_1,\ell_2,\ell_3,\ell_4)$ is a maximal quadruple, then $\Yy_{\ell_1,\ell_3}$ and $\Yy_{\ell_2,\ell_4}$
meet exactly in one point. Such $\bF$-tubes are called {\em orthogonal} if
\bqn
R(\ell_1,\ell_2,\ell_3,\ell_4)=2\,\Id_n
\eqn
(see \cite[Proposition~4.7 and Definition~4.14]{BP}).  
If $\bF=\bR$, the tubes are orthogonal if and only if they are orthogonal as submanifolds of the Riemannian manifold $\snr$.
Given any point $p\in \ioo{\ell_1}{\ell_3}\cup \ioo{\ell_3}{\ell_1}$, 
there exists a unique $\bF$-tube $\Yy_{\ell_2,\ell_4}$ orthogonal to $\Yy_{\ell_1,\ell_3}$
``with endpoint $p$'' in the following sense
\bqn
\begin{cases}
\ell_2:=p\quad&\text{ if } p\in \ioo{\ell_1}{\ell_3}\\
\ell_4:=p\quad&\text{ if } p\in  \ioo{\ell_3}{\ell_1}
\end{cases}
\eqn
In this way we obtain a map
\bqn
\pr_{\Yy_{\ell_1,\ell_3}}:\ioo{\ell_1}{\ell_3}\cup \ioo{\ell_3}{\ell_1}
\to\Yy_{\ell_1,\ell_3}
\eqn
called the {\em orthogonal projection}.

In the case of $\Yy_{0,\infty}$ the map $\pr_{\Yy_{0,\infty}}$ is given by
\bqn
\pr_{\Yy_{0,\infty}}(Y)
=\begin{cases}
\hphantom{-}\imath Y&\text{ if }Y\in \ioo{\ell_0}{\ell_\infty}\\
-\imath Y&\text{ if }Y\in \ioo{\ell_\infty}{\ell_0}\,.
  \end{cases}
\eqn
In view of \eqref{eq:inv-tube}, this implies in particular that the restrictions of $\pr_{\Yy_{\ell_1,\ell_3}}$ to $\ioo{\ell_1}{\ell_3}$ and $\ioo{\ell_3}{\ell_1}$
are both bijective.

\begin{lem}\label{lem:preconseq-of-causal} Assume $(\ell,\ell_1,\ell_2,\ell')$ is maximal.
Then for the pseudodistance $d^1$ on $\snf$ (see \eqref{e.pdist} in \S\ref{subsec:Siegel}) we have,
\bqn
d^1\left(\pr_{\Yy_{\ell,\ell'}}(\ell_1),\pr_{\Yy_{\ell,\ell'}}(\ell_2)\right)=-v(\det R(\ell,\ell_1,\ell_2,\ell'))\,.
\eqn
\end{lem}

\begin{proof}  We may assume that $\ell=\ell_0$
and $\ell'=\ell_\infty$, and set $\ell_j:=\left\langle\begin{pmatrix}Y_j\\ \Id_n\end{pmatrix}\right\rangle$, for $j=1,2$.
Then we have $0\ll Y_1\ll Y_2$ and in particular the eigenvalues $r_1,\dots,r_n$ of the symmetric matrix
$Y_1^{-1/2}Y_2Y_1^{-1/2}$ are all greater than 1.  By invariance of the distance $d^1$ we have
\bqn
\ba
  d^1(\imath Y_1,\imath Y_2)
=&d^1\left(\imath\Id_n,\imath Y_1^{-1/2}Y_2Y_1^{-1/2}\right)
=\sum_{j=1}^n-v(r_j)\\
=&-v\left(\prod_{j=1}^nr_j\right)
=-v\left(\det(Y_1^{-1}Y_2)\right)\\
=&-v(\det R(\ell,\ell_1,\ell_2,\ell'))\,,
\ea
\eqn
where the last equality uses \eqref{eqn:Rinfty}.
This proves the lemma.
\end{proof}

\subsection{Barycenters}\label{subsec:baryc}
Assume that $\bF$ is non-Archimedean.  If the triple of Lagrangians
$(\ell_1,\ell_2,\ell_3)$ is  either maximal or minimal,
then $\ell_2\in\ioo{\ell_1}{\ell_3}\cup\ioo{\ell_3}{\ell_1}$.
We define
\bqn
b(\ell_1,\ell_2,\ell_3):=\pr_{{\Yy}_{\ell_1,\ell_3}}(\ell_2)\in\snf.
\eqn
The \emph{barycenter} of the triple $(\ell_1,\ell_2,\ell_3)$ is the point
\bqn
B(\ell_1,\ell_2,\ell_3)=\pi(b(\ell_1,\ell_2,\ell_3))\,,
\eqn
where $\pi\colon\snf\to\calB_n^\bF$ is the metric quotient introduced in \S\ref{subsec:Siegel}.
Our goal is to show the following (cfr. \cite[Lemma 7.5]{BP}):

\begin{prop}\label{prop: barycenter is symmetric} $B(\ell_1,\ell_2,\ell_3)$ is invariant under permutation of the arguments.
\end{prop}

First we need to establish a formula for the (pseudo-)distance in $\calS_1^\bF$:
\begin{lem}
If $z_1=x_1+\imath y_1,\;z_2=x_2+\imath y_2\in\calS_1^\bF$, then 
$$d(z_1,z_2)=\max\left\{-v\left(\frac{(x_1-x_2)^2}{y_1y_2}\right), -v\left(\frac{y_1}{y_2}\right), -v\left(\frac{y_2}{y_1}\right)\right\}\,,$$
\end{lem}
\begin{proof}
Recall that for a general real closed field $\bF$, if $z_1,z_2\in\calS_1^\bF$,
\bqn
d(z_1,z_2)=\ln\|T+\sqrt{T^2-1}\|_v\,,
\eqn
where $T=1+\frac{|z_1-z_2|^2}{2y_1y_2}$.
Since $\bF$ is non-Archimedean, then for $a,b\in\bF$, $a\geq0$, $b\geq0$, we have
\bqn
\|a+b\|_v=\|\max\{a,b\}\|_v\,.
\eqn
Hence
\bqn
\|T+\sqrt{T^2-1}\|_v=\|T\|_v
\eqn
and
\bqn
d(z_1,z_2)=\ln\|T\|_v\,.
\eqn
Since
\bqn
T=1+\frac{|z_1-z_2|^2}{2y_1y_2}=\frac{(x_1-x_2)^2}{2y_1y_2}+\frac{y_1}{2y_2}+\frac{y_2}{2y_1}\,,
\eqn
then
\bqn
\|T\|_v=\max\left\{\left\|\frac{(x_1-x_2)^2}{y_1y_2}\right\|_v, \left\|\frac{y_1}{y_2}\right\|_v,\left\|\frac{y_2}{y_1}\right\|_v\right\}\,,
\eqn
where we took into account that $\|n\|_v=1$ for any $n\in\bZ\setminus\{0\}$.
\end{proof}

\begin{proof}[Proof of Proposition~\ref{prop: barycenter is symmetric}]
Since $\PSp(2n,\bF)$ acts transitively on maximal and minimal triples,
we may assume that $(\ell_1,\ell_2,\ell_3):=\left(\ell_0,\left\langle\begin{pmatrix}Y\\\Id_n\end{pmatrix}\right\rangle,\ell_\infty\right)$,
where $Y=\pm\Id_n$, depending on whether $(\ell_1,\ell_2,\ell_3)$ is maximal or minimal.
A computation then gives:
\bqn
\ba
\pr_{\Yy_{0,\infty}}(\pm\Id_n)&= \imath\Id_n\\
\pr_{\Yy_{\ell_\infty,\ell_2}}(0)\,\,\,&=(\pm1+ \imath)\Id_n\\
\pr_{\Yy_{\ell_2,\ell_0}}(\infty)\,\,&=\frac{\pm1+ \imath}{2}\Id_n\,.
\ea
\eqn
Observe that if $Z=\diag(z_1,\dots,z_n),W=\diag(w_1,\dots,w_n)\in\snf$, then 
\bqn
d^1(Z,W)=\sum_{i=1}^nd(z_i,w_i)\,.
\eqn
Thus in order to compute the distances between various projections we just have to compute
the following in $\calS_1^\bF$:
\bqn
\ba
d(\imath,\imath\pm 1)&=\ln\max\{\|1\|_v,\|1\|_v,\|1\|_v\}=0\\
d\left(\imath,\left(\frac{1+\imath}{2}\right)\right)&=\ln\max\left\{\left\|\frac12\right\|_v,\left\|\frac12\right\|_v,\|2\|_v\right\}=0\,.
\ea
\eqn
This concludes the proof.
\end{proof}

\section{Applications to maximal framed representations}\label{sec:appl}
In this section we prove Theorem~\ref{thm_intro:thm1}, Corollary~\ref{cor:cor4}, Corollary~\ref{cor:cor2}, 
and Theorem~\ref{thm_intro:thm5}
in the introduction.

\subsection{The  geodesic current $\mu_\rho$ and Theorem~\ref{thm_intro:thm1}}\label{subsec:6.1}
Let $\rho\colon\Gamma\to\PSp(2n,\bF)$ be a maximal framed representation
with maximal framing $\varphi\colon\hg\to\calL(\bF^{2n})$.
Define then for every $(x_1,x_2,x_3,x_4)\in\hg^{[4]}$
\bqn
[x_1,x_2,x_3,x_4]_\rho:=-v(\det R(\varphi(x_1),\varphi(x_2),\varphi(x_3),\varphi(x_4)))\,.
\eqn
It follows from Proposition~\ref{lem:crBP} that $[\,\cdot\,,\,\cdot\,,\,\cdot\,,\,\cdot\,]_\rho$
is a positive crossratio on $\hg$, 
and hence (Proposition~\ref{prop:Rm} and \ref{prop:new4.10})
there is a geodesic current $\mu_\rho$ such that
\bqn
\per(\gamma)=i(\mu_\rho,\delta_c)
\eqn
for every closed geodesic $c$ represented by a hyperbolic element $\gamma\in\Gamma$.

Next, for the computation of $\per(\gamma)$,
recall that $g\in\Sp(2n,\bF)$ is called {\em Shilov hyperbolic} if 
there is a $g$-invariant decomposition $\bF^{2n}=\ell_+\oplus\ell_-$ into Lagrangians
such that all eigenvalues of $g|_{\ell_-}$ have absolute value strictly smaller than one
(and thus all eigenvalues of $g|_{\ell_+}$ have absolute value strictly larger than one),
\cite[Definition~2.12]{BP}. An element $g\in\PSp(2n,\bF)$ is {\em Shilov hyperbolic} if any of its lifts to $\Sp(2n,\bF)$ is. Of course the decomposition is uniquely determined by $g$ and doesn't depend on the lift.
We will need the following:

\begin{lem}[{\cite[Lemma 7.9]{BP}}]\label{lem:6.1}  Let $g\in\Sp(2n,\bF)$ be Shilov hyperbolic and 
let $\lambda_1,\dots,\lambda_n$ be the eigenvalues of $g|_{\ell_+}$.
Then for any $\ell\in\ioo{\ell_-}{\ell_+}$ we have
\bqn
\det R(\ell_-,\ell,g\ell,\ell_+)=\left(\prod_{i=1}^n\lambda_i\right)^2\,.
\eqn
\end{lem}
\begin{proof} We may assume by transitivity of the action of $\Sp(2n,\bF)$ on pairs of transverse Lagrangians, that $(\ell_-,\ell_+)=(\ell_0,\ell_\infty)$.
Furthermore $g=\begin{pmatrix}A&0\\0&{}^tA^{-1}\end{pmatrix}$, 
where $A$ is the matrix of $g|_{\ell_+}$.
Then $\ell$ corresponds to a matrix $X\in\Sym_n(\bF)$ with $X\gg0$
and \eqref{eqn:Rinfty} gives
\bqn
R(\ell_-,\ell,g\ell,\ell_+)=X^{-1}AXA^t\,.
\eqn 
This implies that
\bqn
\det R(\ell_-,\ell,g\ell,\ell_+)=(\det A)^2,
\eqn
 hence the lemma.
\end{proof}

Let $\gamma\in\Gamma$ be hyperbolic with attractive and repulsive fixed points $\gamma_+$ and $\gamma_-$.
It follows from \cite[Theorem~1.9]{BP} that $\rho(\gamma)$ is Shilov hyperbolic with corresponding decomposition
$\bF^{2n}=\varphi(\gamma_-)\oplus\varphi(\gamma_+)$. In particular the framing is uniquely determined by the representation $\rho$.
Furthermore, if $\lambda_1(\gamma),\dots,\lambda_n(\gamma)$ are the eigenvalues of $\rho(\gamma)$
with $|\lambda_1(\gamma)|\geq\dots|\lambda_n(\gamma)|\geq1$,
then Lemma~\ref{lem:6.1} implies that if $x\in\ioo{\gamma_-}{\gamma_+}$,
\bqn
\ba
 \per(\gamma)
&=-v(\det R(\varphi(\gamma_-),\varphi(x),\rho(\gamma)\varphi(x),\varphi(\gamma_+)))\\
&=2\sum_{i=1}^n-v(\lambda_i(\gamma))
=L(\rho(\gamma))\,.
\ea
\eqn
If there is an element $\g\in\Gamma$ with $-v(\tr(\rho(\g)))>0$, then necessarily $\rho(\g)$ has an eigenvalue with the same property and thus $\per(\g)>0$.  Viceversa if $\mu_\rho$ is non-zero, there exists a proper closed rectangle $R=\icc ab\times \icc cd$ with $\mu_\rho(R)>0$. Since $\Gamma$ is a lattice, we find $\g\in\Gamma$ with $\gamma_+\in\ioo bc, \gamma_-\in\ioo da$. For such $\gamma$, the intersection
$i(\mu_\rho,\delta_\gamma)>0$, and thus $\per(\gamma)>0$, which implies that $\|\lambda_1(\gamma)\ldots\lambda_n(\g)\|_v> 1$. Now assume by contradiction that   $-v(\tr(\rho(\g^s)))\leq 0$ for all $s\in\bN$. Then the coefficients of the characteristic polynomial of $\wedge^n\rho(\gamma)$ belong to the ring $\calO:=\{x\in\bF|\, \|x\|_v\leq 1\}$ as well. But $\lambda_1(\gamma)\ldots\lambda_n(\g)\in\bF$  is a root of this monic polynomial, and since $\calO\subset\bF$ is a valuation ring, it is integrally closed in $\bF$ (see \cite[Theorem 3.1.3.(1)]{EP}). This implies $\|\lambda_1(\gamma)\ldots\lambda_n(\g)\|_v\leq 1$, a contradiction, that concludes the proof of Theorem~\ref{thm_intro:thm1}.

\subsection{Displacing representations and Corollary~\ref{cor:cor4}}
Assume that $\Sigma$ is  compact and let $\rho\colon\Gamma\to\PSp(2n,\bF)$
be a maximal framed representation.  If 
\bqn
\Syst(\rho):=\inf_{\gamma\neq e}L(\rho(\gamma))>0\,,
\eqn
it follows from Theorem~\ref{thm_intro:thm1} that  the associated geodesic current $\mu_\rho$ 
has $\Syst(\mu_\rho)>0$; hence there exist $c_1,c_2$ such that, for every $\gamma\in\Gamma$,
\begin{equation}\label{e.6.2}
c_1\ell(\gamma)\leq L(\rho(\gamma))\leq c_2\ell(\gamma)\,,
\end{equation}
where $\ell(\gamma)$ is the hyperbolic length of $\gamma$ \cite[Theorem~1.3]{BIPP1}.
It follows then from \cite{DGLM} that the $\Gamma$-action on $\calB_n^\bF$
induced by $\rho$ is displacing and hence, for every $x\in\calB_n^\bF$,
the map $\gamma\mapsto\rho(\gamma)x$ is a quasi-isometric embedding.
This proves  Corollary~\ref{cor:cor4}.

\subsection{Maximal representations in $\PSp(2n,\bR)$ and Corollary~\ref{cor:cor2}}\label{s.6.2}
If $\Sigma$ is not necessarily compact, the inequalities in  \eqref{e.6.2} do not necessarily hold. However, 
applying \cite[Corollary~1.5 (2)]{BIPP1} to the current $\mu_\rho$, we deduce  following
\begin{cor}\label{cor:PositiveSystole}  
Let $\rho\colon\Gamma\to\PSp(2n,\bF)$ be a representation
admitting a maximal framing defined on $\hg$.
Assume that $\Syst_\Sigma(\rho)>0$.
Then for every 
compact subset $K\subset\Sigma$, there are constants $0<c_1\leq c_2$
such that 
\bqn
c_1\ell(c)\leq L(\rho(\gamma))\leq c_2\ell(c)
\eqn
for every $\gamma\in\Gamma$ representing a closed geodesic $c$ 
contained in $K$.
In particular there exist constants $c_1,c_2$ such that this holds for all $\gamma$ representing simple closed geodesics.
\end{cor}

Assume now that $\bF=\bR$ 
and let $\rho\colon\Gamma\to\PSp(2n,\bR)$ be a maximal representation.
Then there is a maximal framing $\varphi$ defined on $\partial\H$ \cite[Theorem 8]{BIW}
and hence, by Theorem~\ref{thm_intro:thm1}, a geodesic current $\mu_\rho$ with 
\bqn
i(\mu_\rho,\delta_c)=L(\rho(\gamma))\,,
\eqn
for every closed geodesic $c$ represented by a hyperbolic element $\gamma\in\Gamma$.
The Collar Lemma, \cite[Theorem~1.9]{BP},  then implies that
if $\gamma,\eta$ are intersecting hyperbolic elements,
\bqn
\left(e^{\frac{L(\rho(\gamma))}{n}}-1\right)\left(e^{\frac{L(\rho(\eta))}{n}}-1\right)\geq1\,.
\eqn
This  implies that, if $\gamma$ is self-intersecting, 
\bqn
L(\rho(\gamma))\geq n(\ln 2)\,,
\eqn
and  that there are at most $3g-3+p$ conjugacy classes of hyperbolic elements
$\gamma$ with $L(\rho(\gamma))\leq n(\ln2)$.
In particular $\Syst(\mu_\rho)>0$.
Corollary~\ref{cor:cor2} then follows from Corollary~\ref{cor:PositiveSystole}.

\subsection{Lamination type currents and Theorem~\ref{thm_intro:thm5}}
Let $\bF$ be non-Archimedean.
Given a maximal framed representation, as above, and with the notation of \S~\ref{subsec:baryc}, 
define the \emph{barycenter} of $(x,y,z)\in\hg^{(3)}$
\bqn
\beta_\rho(x,y,z):=B(\varphi(x),\varphi(y),\varphi(z))\in\calB_n^\bF\,.
\eqn
It follows from Proposition~\ref{prop: barycenter is symmetric} that $\beta$ is indeed a barycenter according to Definition \ref{d.bary}
and from Lemma~\ref{lem:preconseq-of-causal}
that it is compatible with the crossratio $[\,\cdot\,,\,\cdot\,,\,\cdot\,,\,\cdot\,]_\rho$
defined in \S~\ref{subsec:6.1}.

If now $\mu_\rho$ is of lamination type, we deduce from Theorem~\ref{thm:5.2}
that there is a well defined equivariant isometric embedding
\bqn
\calV({\mu_\rho})\hookrightarrow\calB_n^\bF
\eqn
from the set of vertices $\calV({\mu_\rho})$ of the $\bR$-tree $\calT(\mu_\rho)$
associated to $\mu_\rho$ into the metric space $\calB_n^\bF$.

\subsection{The value group of $[\,\cdot\,,\,\cdot\,,\,\cdot\,,\,\cdot\,]_\rho$ and Theorem~\ref{thm_intro:multicurves} }\label{sec.7.3}
Let $\rho\colon\Gamma\to\Sp(2n,\bF)$ be maximal framed with framing $\varphi\colon\hg\to\calL(\bF^{2n})$,
and set as always
$$
[x_1,x_2,x_3,x_4]_\rho:=-v(\det R(\varphi(x_1),\varphi(x_2),\varphi(x_3),\varphi(x_4))).
$$

\begin{thm}\label{thm:7.3}
Let $\Lambda:=v(\bQ(\rho))$ be the value group of the field $\bQ(\rho)$ generated over $\bQ$ by the matrix coefficents of $\rho$. Then 
$$[x_1,x_2,x_3,x_4]_\rho\in \frac 1{(8n)! }\Lambda$$
for all $(x_1,x_2,x_3,x_4)\in H_\Gamma^{[4]}$.
\end{thm}

\begin{proof}
We might assume that $\bF$ is the real closure of $\bQ(\rho)$. 
Let $(x_1,x_2,x_3,x_4)\in H_\Gamma^{[4]}$ and $\g_i$ hyperbolic with
$(\g_i)_-=x_i$.  If $c_{\g_1}$, $c_{\g_2}$, $c_{\g_3}$, $c_{\g_4}$ are
the characteristic polynomials of $\rho(\gamma_1)$, $\rho(\gamma_2)$,
$\rho(\gamma_3)$, $\rho(\gamma_4)$ then
$c_{\g_i}\in\bQ(\rho)[X]$. 
Since $\bF[i]$ is algebraically closed, $\rho(\g_i)$ splits in  $\bF[i]$. If $\bL$ is the
splitting field in $\bF[i]$
of  $c_{\g_1}c_{\g_2}c_{\g_3}c_{\g_4}\in\bQ(\rho)[X]$ 
then 
$$[\bL:\bQ(\rho)]\leq (8n)! \;.$$
Observe that the field $\bL$ depends on $\rho(\gamma_1)$, $\rho(\gamma_2)$, $\rho(\gamma_3)$, $\rho(\gamma_4)$. 

It is now easy to see that the Lagrangians $\varphi(x_i)\subset \bF^{2n}$ are defined over $\bL\cap\bF$, as a result we can represent them by $\left\langle\begin{pmatrix}X_i\\\Id_n\end{pmatrix}\right\rangle$ with $X_i\in\Sym_n(\bL\cap\bF)$, which implies that  $\det R(\varphi(x_1),\varphi(x_2),\varphi(x_3),\varphi(x_4))\in(\bL\cap\bF)^\times$. We conclude using \cite[XII \S 4 Proposition 12]{Lang} which says that the index of $\Lambda$ in $v((\bL\cap\bF)^\times)$ is at most $(8n)!$. This concludes the proof.
\end{proof}

In particular, if $\bQ(\rho)$ has discrete valuation, we can assume, up
to rescaling the valuation,
that the crossratio $[x_1,x_2,x_3,x_4]_\rho$ is integer valued. Theorem \ref{thm_intro:multicurves} is therefore a direct application of Proposition \ref{p.atomic}.
\section{Examples of maximal framed representations}\label{sec:example}

In this section we collect several interesting examples of maximal framed representations over non-Archimedean real closed fields.
\subsection{Ultralimits of representations and asymptotic cones}\label{ex:RobF} Let $(\rho_k)_{k\geq1}$ be a sequence
of maximal representations into $\Sp(2n,\bR)$ and $\omega$ a non-principal ultrafilter on $\bN$.
This gives rise to a representation $\rho_\omega\colon\Gamma\to\Sp(2n,\bR_\omega)$,
where $\bR_\omega$ is the field of hyperreals and $\rho_\omega(\Gamma)\subset\Sp(2n,\Oo_\sigma)$,
where the infinitesimal $\sigma$ is defined below.
Denoting by $\bR_{\omega,\sigma}$ is the Robinson field, the representation $\rho_{\omega,\sigma}$ obtained by composing $\rho_\omega$
with the projection 
\bqn
\Sp(2n,\bR_\omega)\to\Sp(2n,\bR_{\omega,\sigma})\,,
\eqn
  is a maximal framed representation of $\Gamma$ into
$\Sp(2n,\bR_{\omega,\sigma})$, \cite[Corollary~10.4]{BP}; 
its framing is defined on $\deH$, and Theorem~\ref{thm_intro:thm1} applies.

This construction is closely related to asymptotic cones, as we now recall.
Denoting by $d$ the $\Sp(2n,\bR)$-invariant Riemannian distance on the Siegel $n$-space,
we say that a sequence of scales $(\lambda_k)_{k\in\bN}\in(\bR_{>0})^\bN$ is {\em adapted} (to the sequence $(\rho_k)_{k\in\bN}$) if
for one, and hence every, finite generating set $S\subset\Gamma$
\bq\label{eq:adapted}
\lim_\omega\frac{\max_{\gamma\in S}d\big(\rho_k(\gamma)\imath\Id_n,\imath\Id_n\big)}{\lambda_k}<+\infty\,.
\eq
We obtain then an action
\bqn
^\omega\rho_\lambda\colon\Gamma\to\Isom({}^\omega\calX_\lambda)\,,
\eqn
on the asymptotic cone ${}^\omega\calS_\lambda$ of the sequence of pointed metric spaces given by $(\snr,\imath\Id_n,\frac{d}{\lambda_k})$.

If we set
$\sigma:=(e^{-\lambda_k})_{k\geq1}\in\bR_\omega$, 
then the asymptotic cone ${}^\omega\calX_\lambda$ can be identified
with the metric space $\calB_n^{\bR_{\omega,\sigma}}$ and,
under this identification, ${}^\omega\rho_\lambda$ corresponds to $\rho_{\omega,\sigma}$
(see for example \cite{APcomp}).

\subsection{Maximal representations in $\SL(2,\bF)$}\label{subsec:8.2}

Let $\bF$ be  a real closed field with an order compatible non-Archimedean
valuation, and let $\calT^\bF\supset\calB_1^\bF$ be the $\bR$-tree 
associated to $\SL_2(\bF)$. Then 
$\PP^1(\bF)$ identifies with a subset of $\partial_\infty\calT^\bF$ and
the restriction to $\PP^1(\bF)$ of the crossratio of
$\partial_\infty\calT^\bF$ is the standard crossratio
$[\,\cdot\,,\,\cdot\,,\cdot\,,\,\cdot\,]_\bF$
 in $\PP^1(\bF)$.

Hence any  representation $\rho\colon\Gamma\to\SL(2,\bF)$  
with framing $\varphi\colon\hg\to \PP^1(\bF)$,
gives a framed action on $\calT^\bF$.  
Note that the associated crossratio 
$[x_1,x_2,x_3,x_4]_\varphi
=[\varphi(x_1),\varphi(x_2),\varphi(x_3),\varphi(x_4)]_{\bF}$
is positive if the framing $\varphi$ is maximal.

 Proposition \ref{thm: framed trees} implies the following.

\begin{thm}
\label{thm:SL2 bis}  
Let $\bF$ be  a real closed field with an order
compatible non-Archimedean valuation.
Let $\rho\colon\Gamma\to\SL(2,\bF)$ be a representation 
with a maximal framing $\varphi\colon \hg\to \PP^1(\bF)$.
Denote by  $\mu_\rho$ the geodesic current
associated to the positive crossratio induced by $\varphi$ on 
$ X$.
Then $\mu_\rho$ corresponds to a measured lamination, 
and there is a $\Gamma$-equivariant isometric embedding
\bqn
\calV(\mu_\rho)\hookrightarrow \calT^\bF.
\eqn
In particular, for all hyperbolic $\gamma\in\Gamma$,
$\ell(\rho(\gamma))=i(\mu_\rho,\gamma).$
\end{thm}

\subsection{Unipotent representations of the thrice punctured sphere}\label{ex:strubel}
Let $\Gamma<\PSL(2,\bR)$ be the (unique up to conjugation) lattice such that $\Gamma\backslash\H$ is the thrice punctured sphere.
Then $\Gamma$ admits a presentation
\bqn
\Gamma=\left\<c_1,c_2,c_3:\,c_3c_2c_1\right\>\,,
\eqn
where $c_1,c_2,c_3$ are parabolic elements representing the three inequivalent cusps of $\Gamma$.
Already in this elementary example we are able to illustrate interesting features.
For every $\alpha\in\bR$ we construct maximal framed representations $\rho_\alpha\colon\Gamma\to\Sp(4,\calH(\bR))$,
where $\calH(\bR)$ is the Hahn field with exponents $\bR$ (see Example~\ref{e.real closed}(4)), 
that have the following properties:
\be
\item  for $\alpha\leq1/2$ the corresponding length functions $\gamma\mapsto L(\rho_\alpha(\gamma))$ are
not proportional and hence the corresponding currents $\mu_{\rho_\alpha}$ are distinct in the space 
of projectivized currents;
\item  for $\alpha\in\bQ$ the associated geodesic current $\mu_{\rho_\alpha}$ is a multicurve.
\ee

To this end we use the explicit coordinates obtained by Strubel on the set of $\Sp(2n,\bR)$-conjugacy classes of maximal representations
of $\Gamma$ into $\Sp(2n,\bR)$.
Namely, let 
\bqn
\overline{B}:=\{A\in \GL(n,\bR):\,\mathrm{spec} (A)\subset\{z\in\bC:\,|z|\leq1\}\}
\eqn
and 
\bqn
R=\{(X_1,X_2,X_3)\in\overline{B}^3:\, X_3\,{}^tX_2^{-1}X_1\text{ is symmetric positive definite}\}\,.
\eqn
Then, given $X:=(X_1,X_2,X_3)\in R$, the formulas
\bqn
\ba
\rho_X(c_1)&=\begin{pmatrix} X_1&0\\X_1+X_2^{-1}{}^tX_3&{}^tX_1^{-1}\end{pmatrix}\\
\rho_X(c_3)&=\begin{pmatrix} {}^tX_3^{-1}&-{}^tX_3^{-1}-X_1^{-1}{}^tX_2\\0&X_3\end{pmatrix}
\ea
\eqn
determine a representation $\Gamma\to\Sp(2n,\bR)$ that is maximal \cite[Theorem~2]{Strubel}.
Moreover every maximal representation of $\Gamma$ into $\Sp(2n,\bR)$ is conjugate to a $\rho_X$ for $X\in R$
and $\Sp(2n,\bR)$-conjugacy classes of maximal representations correspond to $\mathrm{O}(n)$-conjugacy classes in $R$
for the diagonal conjugation action of $\mathrm{O}(n)$.

If now $\bF$ is a real closed field and 
\bqn
\overline{B}_\bF=\big\{A\in\GL(n,\bF):\,\mathrm{spec}(A)\subset\{z\in\bF(\sqrt{-1}):\,|z|\leq1\}\big\}
\eqn
and 
\bqn
R_\bF:=\{(X_1,X_2,X_3)\in\overline{B}_\bF^3:\, X_3\,{}^tX_2^{-1}X_1\text{ is symmetric positive definite}\}\,,
\eqn
the following gives a source of maximal framed representations over any real closed field $\bF$:

\begin{prop}\label{prop:8.2}  For every $X:=(X_1,X_2,X_3)\in R_\bF$ the formulas for $\rho_X(c_1)$ and $\rho_X(c_3)$ define 
a maximal framed representation $\rho_X\colon\Gamma\to\Sp(2n,\bF)$.
\end{prop}
The proof is beyond the scope of this paper.  Let us just mention that it is an easy application of the Tarski--Seidenberg principle
(see \cite[Proposition~5.1.3]{BCR}) and the fact that $R$ parametrizes a semi-algebraic subset
of $\Hom(\Gamma,\Sp(2n,\bR))$.

We are interested in the case in which $\rho_X(c_1)$,  $\rho_X(c_2)$ and $\rho_X(c_3)$ are all unipotent,
which is equivalent to $X_1,-X_2,X_3$ being unipotent.  
This is never the case if $n=1$ as one can see from the above formulas,
while already for $\Sp(4,\bR)$ there are interesting examples with unipotent boundary holonomy.
We restrict here to the subset of $R$ consisting of triple $(X_1,X_2,X_3)$
such that $X_1, -X_2, X_3$ are unipotent and $X_3{}^tX_2^{-1}X_1=\Id$.  The quotient by $\mathrm{O}(2)$-conjugation of such triples can be parametrized by 
\bqn
\left\{
\begin{pmatrix}1&\frac{4}{x}\\0&1\end{pmatrix},
\begin{pmatrix}-3+y&-x\\ \frac{(y-2)^2}{x}&1-y\end{pmatrix},
\begin{pmatrix}1+y&\frac{y^2}{x}\\ -x&1-y\end{pmatrix}, x>0, y\in\bR\right\}\,.
\eqn
The corresponding matrices are then
\bqn
\ba
\rho_X(c_1)=&
\begin{pmatrix} X_1&0\\X_1+X_2^{-1}{}^tX_3&{}^tX_1^{-1}\end{pmatrix}=
\begin{pmatrix}
	1&\frac{4}{x}&0&0\\
	0&1&0&0\\
	2&\frac{4}{x}&1&0\\
	-\frac{4}{x}&2&-\frac{4}{x}&1
\end{pmatrix}\\
\rho_X(c_3)=&
\begin{pmatrix} {}^tX_3^{-1}&-{}^tX_3^{-1}-X_1^{-1}{}^tX_2\\0&X_3\end{pmatrix}=
\begin{pmatrix}
	1-y&x&-2&-x-\frac{y^2}{x}\\
	-\frac{y^2}{x}&1+y&x+\frac{y^2}{x}&-2\\
	0&0&1+y&\frac{y^2}{x}\\
	0&0&-x&1-y
\end{pmatrix}\,.
\ea
\eqn
The above formulas allow us to consider $\rho_X$ as a representation of $\Gamma$ with coefficients 
in the ring $\bR[x,\frac{1}{x},y]$.
Now for every $\alpha\in\bR$ we consider the ring morphism of $\bR[x,\frac{1}{x},y]$ into the Hahn field $\calH(\bR)$
defined by sending $x$ to $x$ and $y$ to $x^\alpha$.  In this way we obtain for every $\alpha\in\bR$
a representation $\rho_\alpha\colon\Gamma\to\Sp(4,\mathcal{H}(\bR))$.  
It is easy to verify that the triple $X=(X_1,X_2,X_3)$ with $y=x^\alpha$ is in $R_{\calH(\bR)}$
and it follows from Proposition~\ref{prop:8.2} that $\rho_\alpha$ is maximal framed.    
For the computation of the length function $L$
it is not difficult to see that if $g\in\Sp(4,\bF)$, where $\bF$ is real closed non-Archimedean with an order compatible valuation, then 
\bqn
L(g)=-v(T(g))\,.
\eqn
where $T(g)=(\tr g)^2-\tr g^2-4$.
In our case we obtain
\bqn
T(\rho_\alpha(c_1^{-1}c_3))
=4(4x^2+32x^{-4+4\alpha}+(18+8x^{2\alpha})+4x^{-2}(16+12x^{2\alpha}+x^{4\alpha}))\,
\eqn
and
\bqn
T(\rho_\alpha(c_1^{-1}c_2))
=4(50+4x^2+8x^{2\alpha}+4x^{-2}(16+12x^{2\alpha}+x^{4\alpha}))
\eqn
so
\bqn
v(T(\rho_\alpha(c_1^{-1}c_3)))=\min(-2,-4+4\alpha,-2+2\alpha)
=\begin{cases}
-2&\alpha\geq\frac12\\
-4+4\alpha&\alpha\leq\frac12
\end{cases}
\eqn
and
\bqn
v(T(\rho_\alpha(c_1^{-1}c_2)))=\min(-2,-2+4\alpha,-2+2\alpha)
=\begin{cases}
-2&\alpha\geq0\\
-2+4\alpha&\alpha\leq0\,.
\end{cases}\,
\eqn
We deduce that for $\alpha\leq1/2$ the length functions $\gamma\mapsto L(\rho_\alpha(\gamma))$ are distinct
even when considered up to positive scaling.  It is easy to verify that 
\bqn
\bQ(\rho_\alpha)=\begin{cases}
\bQ(x^\alpha)&\text{ if } \alpha\in\bQ\smallsetminus\{0\}\\
\bQ(x)&\text{ if } \alpha=0\\
\bQ(x,x^\alpha)&\text{ if } \alpha\notin\bQ\,.\\
\end{cases}
\eqn
As a result, the image of the valuation is
\bqn
v(\bQ(\rho_\alpha))=\begin{cases}
\alpha\bZ&\text{ if }\alpha\in\bQ\smallsetminus\{0\}\\
\bZ&\text{ if }\alpha=0\\
\bZ+\alpha\bZ&\text{ if }\alpha\notin\bQ\,,\\
\end{cases}
\eqn
which implies by Theorem~\ref{thm_intro:multicurves} that the geodesic current corresponding to $\rho_\alpha$ is a multicurve if $\alpha\in\bQ$.


\end{document}